\documentclass[a4paper]{article}

\usepackage[french,english]{babel}
\usepackage{amsfonts}
\usepackage{amsmath,amssymb,amscd}
\usepackage{mathrsfs}  
\usepackage{amsthm}
\usepackage{a4wide}   
\usepackage[T1]{fontenc} 
\usepackage{newlfont} 
\usepackage{color}
\usepackage{stackrel}

\usepackage{enumitem}  
\usepackage[colorinlistoftodos]{todonotes}

\usepackage[backref=page,bookmarksopen=true,colorlinks=true, linkcolor={blue!50!black}, pdfhighlight=/O, ocgcolorlinks=true]{hyperref}

\usepackage{multicol}

\usepackage[a4paper, total={6in,9.1in}]{geometry}

\usepackage{tikz-cd}
\usepackage{tikz}
\usetikzlibrary{shapes}
\usetikzlibrary{decorations.pathmorphing}
\usetikzlibrary{arrows}
\usetikzlibrary{decorations.markings}
\usetikzlibrary{backgrounds,calc}
\tikzset{middlearrow/.style={
        decoration={markings,
            mark= at position 0.6 with {\arrow{#1}} ,
        },
        postaction={decorate}
    }
}
\usepackage{capt-of}
\usepackage{environ}
\makeatletter
\newsavebox{\measure@tikzpicture}
\NewEnviron{scaletikzpicturetowidth}[1]{t%
  \def\tikz@width{#1}%
  \begin{lrbox}{\measure@tikzpicture}%
  \BODY
  \end{lrbox}%
  \pgfmathparse{#1/\wd\measure@tikzpicture}%
  \BODY
}
\makeatother

\usepackage{comment}

\newcommand{\cache}[1]{}

\usepackage[all]{xy}  

\usepackage[applemac]{inputenc}       

\usepackage[affil-it]{authblk}

\numberwithin{equation}{section}

\theoremstyle{plain}

\newtheorem{theo}[equation]{Theorem}

\newtheorem{prop}[equation]{Proposition}
\newtheorem{lemm}[equation]{Lemma}
\newtheorem{coro}[equation]{Corollary}

\theoremstyle{definition}

\newtheorem{defi}[equation]{Definition}
\newtheorem{exem}[equation]{Example}

\newtheorem{rema}[equation]{Remark}

\newtheorem*{exem*}{Example}
\newtheorem*{exems*}{Examples}
\newtheorem*{exam*}{Example}
\newtheorem*{exams*}{Examples}
\newtheorem*{rema*}{Remark}
\newtheorem*{remas*}{Remarks}

\theoremstyle{definition}
\newtheorem*{defi*}{Definition}
\newtheorem*{defiprop*}{Definition-Proposition}

\theoremstyle{plain}
\newtheorem*{prop*}{Proposition}
\newtheorem*{lemm*}{Lemma}
\newtheorem*{coro*}{Corollary}
\newtheorem*{theo*}{Theorem}

 \def\cdr@enoncedef{%
 \newenvironment{enonce*}[2][plain]%
 {\let\cdrenonce\relax \theoremstyle{##1}%
 \newtheorem*{cdrenonce}{##2}%
 \begin{cdrenonce}}%
 {\end{cdrenonce}}   }%

 \AtBeginDocument{%
    \cdr@enoncedef}

\def\cf{{\it cf.\/}\ }
\def\ie{{\it i.e.\/}\ }

\def\truc{\unskip\kern 3pt\penalty 500
\hbox{\vrule\vbox to 5pt{\hrule width 4pt\vfill\hrule}\vrule}\kern
3pt}

\def\vect{\overrightarrow}

\def\N{{\mathbb N}}    
\def\Z{{\mathbb Z}}

\def\R{{\mathbb R}}

\def\A{{\mathbb A}}
\def\M{{\mathbb M}}

\newcommand{\g}[1]{\mathfrak{#1}} 

\def\qa{\alpha}     

\def\qf{\varphi}

 \def\ql{\lambda}

\def\qn{\nu}

\def\qp{\pi}
\def\qr{\rho}

\def\QF{\Phi}

\def\sha{{\mathcal A}}   

\def\shm{{\mathcal M}}

\def\sht{{\mathcal T}}

\def\SHI{{\mathscr I}}

\def\SHT{{\mathscr T}}

\newcommand{\stephaneline}[1]{\todo[inline,color=green!40]{#1}}
\newcommand{\tristan}[1]{\todo{#1}}

\begin{document}

\title{MV Polytopes and Masures}

\author[]{
	Tristan Bozec\thanks{IMAG, Universit\'e Montpellier, Montpellier, France \\
											\href{mailto:tristan.bozec@umontpellier.fr}{tristan.bozec@umontpellier.fr}}, 
	St\'ephane Gaussent\thanks{ICJ, Universit\'e Jean Monnet, Saint-\'Etienne, France \\
											\href{mailto:stephane.gaussent@univ-st-etienne.fr}{stephane.gaussent@univ-st-etienne.fr}}}

\date{}

\maketitle


\begin{abstract} {We realize affine Mirkovi\'c--Vilonen polytopes using Littelmann's path model in the framework of masures. We are also able to read the decorations on the paths in the case of $\widehat{\mathfrak{sl}}_2$.
}
\end{abstract}

\setcounter{tocdepth}{2}    
\tableofcontents

\section*{Introduction}

Given a finite type Lie algebra $\mathfrak g$, Mirkovi\'c--Vilonen and Anderson~\cite{MR1958098} have defined the so-called {\it MV polytopes} as the image by the moment map of cycles in the affine Grassmanniann. They provide a realization of the combinatorial Kashiwara crystals associated with $\mathfrak g$. These polytopes have been recovered in several ways: via preprojective algebras (Baumann--Kamnitzer~\cite{MR2892443}), Poincar\'e--Birkhoff--Witt (PBW) bases of the half quantum group $U_q^+(\mathfrak g)$ (Lusztig, Kamnitzer~\cite{MR2630039}), diagrammatic Khovanov--Lauda--Rouquier (KLR) algebras (Tingley--Webster~\cite{MR3542489}), or through affine buildings and galleries (Ehrig~\cite{E10}). It is interesting to wonder how  these constructions extend to the affine setting. Baumann--Kamnitzer--Tingley~\cite{BKT14} used the point of view of preprojective algebras to naturally define \emph{affine MV polytopes}, recovered later by PBW (resp.\ KLR) methods by Muthiah--Tingley~\cite{MR3874704} (resp.\ Tingley--Webster, \textit{op.\!\! cit.}). The present paper explains how these affine MV polytopes can be constructed using Littelmann's paths adapted to the context of Gaussent--Rousseau masures~\cite{GR08} which generalize buildings. Precisely, we use retractions to obtain the non-decorated polytopes in any affine type, and give a method to recognize the partitions decorating the $\widehat{\mathfrak{sl}}_2$ MV polytopes.

In the first section we recall the notions we need about masures. Then in the second we give a few definitions regarding LS-paths an prove a few combinatorial lemmas about their crystal structure. In the third section we obtain the first main result of this paper, Theorem~\ref{retractdense}, relating these crystal-type data to retractions of paths in masures. We prove this statement by induction using parabolic retractions. In section 4, we introduce affine MV polytopes, in the $\widehat{\mathfrak{sl}}_2$ case, essentially following~\cite{BDKT12}. In \S 4.2, we link this section to the previous ones, stating that we recover the bottom part of the undecorated polytope associated to a path using retractions. We then prove combinatorial and technical lemmas used in the last section to prove our second main result obtained in section 5, Theorem~\ref{decopath}: these decorations, which are partitions, can be recognized on a specific class of paths, by simple examination. The combination of our results allows one to recover these decorations for \emph{any} path in the crystal.

It is of course natural to wonder how to retrieve the decorations defined in~\cite{BKT14} for any affine type, as has been done for instance in the above mentioned~\cite{MR3542489}. We believe that one could achieve this program by using the $\widehat{\mathfrak{sl}}_2$ case as an elementary step since the associated polytopes are the 2-faces of arbitrary affine MV polytopes.
To this end, the notion of zigzag defined in~\ref{defzigzag} should easily be generalized.

\subsection*{Acknowledgements}
The first  author has received funding from the European Research Council (ERC) under the
European Union's Horizon 2020 research and innovation programme (Grant Agreement No.\ 768679). We are thankful to Pierre Baumann for fruitful discussions.

\section{Recollections}
\label{se:Recollection}


\subsection{The vectorial data}
\label{suse:Vectorial}  

Let $\M=(\qa_j(\qa_i^\vee))_{i,j\in I}$ be a Kac-Moody matrix, {\it i.e.} a generalized Cartan matrix, meaning a matrix with non positive integers coefficients, $2$'s on the diagonal and with the symmetry of $0$'s.

We consider a root generating system $(\M,X,Y,(\qa_i)_{i\in I}, (\qa^\vee_i)_{i\in I})$ where $\M$ is Kac-Moody matrix, $X$ and $Y$ are two dual free $\mathbb Z$-modules of finite rank, $I$ a finite set, $(\qa^\vee_i)_{i\in I}$ a family in $Y$ and $(\qa_i)_{i\in I}$ a  family in the dual $X$.
We suppose these families free, \ie the sets  $\{\qa_i\mid i\in I\}$ and $\{\qa^\vee_i\mid i\in I\}$ are linearly independent. Further, we assume that $\M_{i,j} = \qa_j(\qa_i^\vee)$.

Let $V = Y\otimes \mathbb R$, then every element of $X$ defines a linear form on $V$ and the formula $s_i(v)=v-\qa_i(v)\qa_i^\vee$ defines a linear involution in $V$. The subgroup generated by the $s_i$ is $W^v$, the Weyl group of the corresponding Kac-Moody Lie algebra $\g g_\M$ and the associated real root system is
\[
\QF=\{w(\qa_i)\mid w\in W^v,i\in I\}\subset Q=\bigoplus_{i\in I}\,\Z.\qa_i.
\] We consider also the dual action of $W^v$ on $V^*$.

We set $\QF^\pm{}=\QF\cap Q^\pm{}$ where $Q^\pm{}=\pm{}(\bigoplus_{i\in I}\,(\Z_{\geq 0}).\qa_i)$. 
Also $Q^\vee:=\bigoplus_{i\in I}\,\Z.\qa_i^\vee$ and $Q^\vee_\pm{}=\pm{}(\bigoplus_{i\in I}\,(\Z_{\geq 0}).\qa_i^\vee)$. We have  $\QF=\QF^+\cup\QF^-$ and, for $\qa=w(\qa_i)\in\QF$, $s_\qa=ws_iw^{-1}$ and $s_\qa(v)=v-\qa(v)\qa^\vee$, where the coroot $\qa^\vee=w(\qa_i^\vee)$ depends only on $\qa$.

The set  $\QF_{}$ is an (abstract) reduced real root system in the sense of \cite{MP89}, \cite{MP95} or \cite{BP96}. We will use imaginary roots: $\QF_{im}=\QF^+_{im}\sqcup\QF^-_{im}$ with 
 $-\QF^-_{im}=\QF^+_{im}\subset Q^+$,  $W^v-$stable.
The set  $\QF_{all}=\QF_{}\sqcup\QF_{im}$ of all roots has to be an (abstract) root system in the sense of \cite{BP96}.
An example for $\QF_{all}$ is the full set of roots of $\g g_\M$.

The fundamental positive chamber is $C^v_f=\{v\in V\mid\qa_i(v)>0,\forall i\in I\}$. Its closure $\overline{C^v_f}$ is the disjoint union of the vectorial faces $F^v(J)=\{v\in V\mid\qa_i(v)=0,\forall i\in J,\qa_i(v)>0,\forall i\in I\setminus J\}$ for $J\subset I$. We set $V_0 = F^v(I){=V^{W^v}}$.

The positive (resp. negative) vectorial faces are the sets $w.F^v(J)$ (resp. $-w.F^v(J)$) for $w\in W^v$ and $J\subset I$. The support of such a face is the vector space it generates. The set $J$ or the face $w.F^v(J)$ or an element of this face is called spherical if the group $W^v(J)$ generated by $\{s_i\mid i\in J\}$ is finite. An element of a vectorial chamber $\pm w.C^v_f$ is called regular.

The Tits cone  $\sht$ is the (disjoint) union of the positive vectorial faces. It is a $W^v-$stable convex cone in $V$. Actually $W^v$ permutes the vectorial walls $M^v(\qa)=\ker(\qa)$ (for $\qa\in\QF$); it acts simply transitively on the positive (resp. negative) vectorial chambers.


\subsection{The model apartment}
\label{suse:Apart} 
As in \cite[1.4]{R11} the model apartment $\A$ is $V$ considered as an affine space and endowed with a family $\shm$ of walls. These walls are the affine hyperplanes directed by $\ker(\qa)$:
\[
M(\qa,k)=\{v\in V\mid\qa(v)+k=0\}\qquad\text{for }\qa\in\QF_{}\text{ and } k\in\Z.
\]

For $\qa=w(\qa_i)\in\QF$, $k\in \Z$ and $M=M(\qa,k)$, the reflection $s_{\qa,k}=s_M$ with respect to $M$ is the affine involution of $\A$ with fixed points the wall $M$ and associated linear involution $s_\qa$. In equation, this gives for any $x\in\mathbb A$,
\[
s_{\qa,k} (x) = x - (\qa(x) + k)\qa^\vee.
\]

The affine Weyl group $W^a$ is the group generated by the reflections $s_M$ for $M\in \shm$; we assume that $W^a$ stabilizes $\shm$.
We know that $W^a=W^v\ltimes Q^\vee$; here $ Q^\vee$ has to be understood as groups of translations.

An automorphism of $\A$ is an affine bijection $\qf:\A\to\A$ stabilizing the set of pairs $(M,\qa^\vee)$ of a wall $M$ and the coroot $\qa^\vee$ associated with $\qa\in\QF_{}$ such that $M=M(\qa,k)$, $k\in \Z$. 
   We write $\vect\qf:V\to V$ the linear application associated to $\qf$.
   The group $Aut(\A)$ of these automorphisms contains $W^a$ and normalizes it.
We consider also the group $Aut^W_\R(\A)=\{\qf\in Aut(\A)\mid\vect{\qf}\in W^v\}$ of vectorially-Weyl automorphisms.
One has $Aut^W_\R(\A)=W^v\ltimes P^\vee$, where $ P^\vee=\{ v\in V\mid \qa(v)\in \Z, \forall\qa\in\QF \}$.

For $\qa\in{\QF_{all}}$ and $k\in\R$, $D(\qa,k)=\{v\in V\mid\qa(v)+k\geq 0\}$ is an half-space. It is called a half-apartment if $k\in \Z$ and $\qa\in\QF$. We write  $D(\alpha,\infty) = \mathbb A$.

The Tits cone $\mathcal T$ is convex and $W^v-$stable cones, therefore, we can define a $W^v-$invariant preorder relation on $\mathbb A$: 
\[
x\leq y\;\Leftrightarrow\; y-x\in\mathcal T.
\]

\subsection{Faces and sectors} 
\label{suse:Faces}

The faces in $\mathbb A$ are associated to the above systems of walls
and half-apartments. As the set of walls might be dense in $\mathbb A$, seen as a finite dimensional real vector space, the faces are no longer subsets, but filters of subsets of $\mathbb A$. For the definition of that notion and its properties, we refer to \cite{BT72} or \cite{GR08}.

If $F$ is a subset of $\mathbb A$ containing an element $x$ in its closure,
the germ of $F$ in $x$ is the filter $\mathrm{germ}_x(F)$ consisting of all subsets of $\mathbb A$ which contain intersections of $F$ and neighbourhoods of $x$. We say that $x$ is the origin of this germ.
In particular, if $x\neq y\in \mathbb A$, we denote the germ in $x$ of the segment $[x,y]$ by $[x,y)$.
For $y\neq x$, the segment germ $[x,y)$ is called of sign $\pm$ if $y-x\in\pm\sht$.
The segment $[x,y]$ or the segment germ $[x,y)$ is called preordered if $x\leq y$ or $y\leq x$.

Given $F$ a filter of subsets of $\mathbb A$, its enclosure $cl_{\mathbb A}(F)$ (resp. {\it closure} $\overline F$) is the filter made of the subsets of $\mathbb A$ containing an element of $F$ of the shape $\cap_{\alpha\in\QF_{all}}D(\alpha,k_\alpha)$, where $k_\alpha\in \mathbb Z\cup\{\infty\}$ (resp. containing the closure $\overline S$ of some $S\in F$).

\medskip

A local face $F$ in the apartment $\mathbb A$ is associated
 to a point $x\in \mathbb A$, its vertex or origin, and a  vectorial face $F^v{=:\vect F}$ in $V$, its direction. It is defined as $F=germ_x(x+F^v)$.
 Its closure is $germ_x(x+\overline{F^v})$. Its sign is the sign of $F^v$.

There is an order on the local faces: the assertions ``$F$ is a face of $F'$ '',
``$F'$ covers $F$ '' and ``$F\leq F'$ '' are by definition equivalent to
$F\subset\overline{F'}$.
 The dimension of a local face $F$ is the smallest dimension of an affine space generated by some $S\in F$.
  The (unique) such affine space $E$ of minimal dimension is the support of $F$; if $F=germ_x(x+F^v)$, $supp(F)=x+supp(F^v)$.
 A local face is spherical if the direction of its support meets the open Tits cone (\ie if $F^v$ is spherical), then its pointwise stabilizer $W_F$ in $W^a$ is finite.

\medskip
 A local chamber is a maximal local face, \ie a local face $germ_x(x \pm w.C^v_f)$ for $x\in\A$ and $w\in W^v$.
 The fundamental local chamber of sign $\pm$ is $C_0^\pm=germ_0(\pm C^v_f)$.
A (local) panel is a spherical local face maximal among local faces which are not chambers, or, equivalently, a spherical face of dimension $n-1$. Its support is a wall. Sometimes the adjective ``local'' will be dropped out the notation.

\medskip
 A sector in $\mathbb A$ is a $V-$translate $\mathfrak s=x+C^v$ of a vectorial chamber
$C^v=\pm w.C^v_f$, $w \in W^v$. The point $x$ is its {\it base point} and $C^v{=:\vect{\g s}}$ its direction.  Two sectors have the same direction if, and only if, they are conjugate
by $V-$translation, and if, and only if, their intersection contains another sector.

The sector-germ of a sector $\mathfrak s=x+C^v$ in $\mathbb A$ is the filter $\mathfrak S$ of
subsets of~$\mathbb A$ consisting of the sets containing a $V-$translate of $\mathfrak s$, it is well
determined by the direction $C^v{=\vect{\g s}=:\vect{\g S}}$. So, the set of
translation classes of sectors in $\mathbb A$, the set of vectorial chambers in $V$ and
 the set of sector-germs in $\mathbb A$ are in canonical bijection.
  We denote the sector-germ associated to the  fundamental vectorial chamber $\pm C^v_f$ by $\g S_{\pm\infty}$.

 A sector-face in $\mathbb A$ is a $V-$translate $\mathfrak f=x+F^v$ of a vectorial face
$F^v=\pm w.F^v(J)$. The sector-face-germ of $\mathfrak f$ is the filter $\mathfrak F$ of
subsets containing a translate $\mathfrak f'$ of $\mathfrak f$ by an element of $F^v$ ({\it i.e.} $\mathfrak
f'\subset \mathfrak f$). If $F^v$ is spherical, then $\mathfrak f$ and $\mathfrak F$ are also called
spherical. The sign of $\mathfrak f$ and $\mathfrak F$ is the sign of $F^v$.

 \subsection{The masure}
 \label{suse:Masure}

In this section, we recall some properties of the masure as defined in \cite{R16} and in \cite{He21}.

\medskip

An apartment of type $\mathbb A$ is a set $A$ endowed with a set $Isom^W\!(\mathbb A,A)$ of bijections, called Weyl-isomorphisms, such that, if $f_0\in Isom^W\!(\mathbb A,A)$, then $f\in Isom^W\!(\mathbb A,A)$ if, and only if, there exists $w\in W^a$ satisfying $f = f_0\circ w$.
A Weyl-isomorphism between two apartments $\varphi :A\to A'$ is a bijection such that, for any $f\in Isom^W\!(\mathbb A,A)$, $f'\in Isom^W\!(\mathbb A,A')$, we have $f'^{-1}\circ\qf\circ f\in W^a$. The set of these isomorphisms is written $Isom^W(A,A')$. Thanks to these isomorphisms, faces, local faces, sectors, Tits cone... are defined in any apartment of type $\mathbb A$.
 
 \medskip
A masure of type $\mathbb A$ is a set $\SHI$ endowed with a covering $\mathcal A$ of subsets called apartments, each endowed with some structure of an apartment of type $\A$.
 We do not recall here the precise definition, which was simplified by H\'ebert \cite{He21}. We indicate some of its main properties:

\smallskip
\par {\bf a)} If $F$ is a point, a preordered segment,  a local face or a spherical sector face in an apartment $A$ and if $A'$ is another apartment containing $F$, then $A\cap A'$ contains the
enclosure $cl_A(F)$ of $F$ and there exists a Weyl-isomorphism from $A$ onto $A'$ fixing $cl_A(F)$.

A filter or subset in $\SHI$ is called a preordered segment, a preordered segment germ, a local face, a spherical sector face or a spherical sector face germ if it is included in some apartment $A$ and is called like that in $A$.

\smallskip
\par {\bf b)} If $\mathfrak F$ is  the germ of a spherical sector face and if $F$ is a face or a germ of a spherical sector face, then there exists an apartment that contains $\mathfrak F$ and $F$.

\smallskip
\par {\bf c)}  If two apartments $A,A'$ contain $\mathfrak F$ and $F$ as in {\bf b)}, then their intersection contains $cl_A(\mathfrak F\cup F)$ and there exists a Weyl-isomorphism from $A$ onto $A'$ fixing $cl_A(\mathfrak F\cup F)$.

\smallskip
\par {\bf d)} We consider the relation $\le$ on $\SHI$ defined as follows:
\[
x\le y \iff \exists A\in\sha \text{ such that }x,y\in A\text{ and } x\le_A y 
\]
then $\le$ is a well defined preorder, in particular transitive, that extends the preorder on $\A$ given by the Tits cone.


 \subsection{The group}
 \label{suse:Group}

Set $\mathscr K = \mathbb C(\!(t)\!)$ and denote by $\mathscr O = \mathbb C[\![t]\!]$ its ring of integers and by $val$ the associated discrete valuation.
Let $G = G(\mathscr K)$ be a split Kac-Moody group over $\mathscr K$, with Lie algebra $\g g_\M$, introduced in \ref{suse:Vectorial}. We denote by $T = T(\mathscr K)$ the maximal torus such that the $\mathbb Z$-lattice of coweights $Hom (\mathcal K^*, T)$ is $Y$, and the dual $\mathbb Z$-lattice of weights $Hom (T,\mathcal K^*)$ is $X$. Let $B = B^+ = B(\mathscr K)$ be the Borel subgroup associated to the choice of simple roots made in \ref{suse:Vectorial} and let $B^-$ the opposite Borel subgroup in $G$.
The set of real roots of $(G,T)$ is $\Phi$ and the $\mathbb Z$-lattice of coroots is $Q^\vee$. Finally the Weyl group of $(G,T)$ identifies with $W^v$.

To each real root $\alpha$ corresponds a subgroup $U_\alpha = U_\alpha (\mathscr K)$ isomorphic to $(\mathscr K, +)$, via $x_\alpha : \mathscr K\to U_\alpha$. Then the Borel subgroups decomposes as $B^\pm = TU^\pm$, for $U^\pm$ generated by the subgroups $U_\alpha$ in $G$, for $\alpha\in\Phi^\pm$. And to any $k\in\mathbb Z$, one defines $U_{\alpha,k} = x_\alpha(\{u\in\mathscr K, val (u)\geqslant k\})$, a subgroup of $U_\alpha$.

\medskip
We consider now the masure $\SHI = \SHI (G,\mathscr K)$ as defined in \cite{R16}. It satisfies the properties listed in \ref{suse:Masure}. And even better, the group $G$ acts upon it in a such a way that all apartments are given by this action, for any $A\in\mathcal A$, there exists $g\in G$, $A = g\cdot \mathbb A$. Further, one has the following properties for some fixators (pointwise stabilizers):
\begin{itemize}
\item[$-$] Fix$_G(\mathbb A) = T(\mathscr O)$;
\item[$-$] Fix$_G(D(\alpha,k)) = T(\mathscr O)U_{\alpha,k}$ and all the apartments containing $D(\alpha,k)$ are given by the action of an element of the form $x_\alpha (u)$, $val (u)\geqslant k$;
\item[$-$] Fix$_G(\g S_{\pm\infty}) = T(\mathscr O)U^\pm$;
\item[$-$] Fix$_G(\{0\}) = G(\mathscr O)$.
\end{itemize}
\noindent For any subset or filter of subsets $\Omega$ of $\SHI$, we denote the fixator of $\Omega$ by $G_\Omega$, in the case of $\Omega = \{x\}$, we just write $G_x$.

\medskip
If $N$ is the stabilizer of $\A$ in $G$, there exists an homomorphism $\qn:N\to Aut(\A)$ such that the image group $W_Y= W^v\rtimes Y$ permutes the walls, local faces, sectors, sector-faces... and contains the affine Weyl group $W^a=W^v\ltimes Q^\vee$ \cite[4.13.1]{R13}. 
The group $T$ acts by translations as follows: for $t\in T$, $\nu(t)$ is defined by $\chi(\nu(t)) = - val \chi(t)$, for all $\chi\in X$. For example, an element $\mu(t)\in T$, for $\mu\in Q^\vee$ will act by the translation by $-\mu$.

 \subsection{The local behaviour of the action}\label{suse:Local}

Let $x$ be a point in the standard apartment $\mathbb A$. Let $\Phi_x$ be the set of all roots $\alpha$ such that $\alpha (x)\in \Z$. It is a closed subsystem of roots. Its associated Weyl group $W^v_x$ is a Coxeter group.

We have twinned buildings $\SHI^+_x$ and $\SHI^-_x$ whose elements are segment germs $[x,y)$ for $y\in\SHI$, $y\not=x$, $y\geq{}x$ and $y\leq{}x$, respectively.
 We consider their unrestricted structure, so the associated Weyl group is $W^v$ and the chambers (resp. closed chambers) are the local chambers $C=germ_x(x+C^v)$ (resp. local closed chambers $\overline{C}=germ_x(x+\overline{C^v})$), where $C^v$ is a vectorial chamber, \cf \cite[4.5]{GR08} or \cite[{\S{}} 5]{R11}. 

To $\A$ is associated a twin system of apartments $\A_x = (\A_x^-,\A_x^+)$. So that the walls and half-apartments of $\mathbb A$ give walls and half-apartments in the twinned apartment $\A_x = (\A_x^-,\A_x^+)$.


Let $\SHI_x$ be the union of $\SHI^+_x$ and $\SHI^-_x$ in $\SHI$.
The group $\overline G_x = G_x/G_{\SHI_x}$ acts on $\SHI_x^+$ and $\SHI_x^-$.
For any root $\alpha\in \Phi_x$ with $\alpha(x) = k\in\Z$, the group $\overline U_\alpha = U_{\alpha, k}/U_{\alpha, k +1}$ is a subgroup of $\overline G_x$ that can be identified set-theoretically with the elements $x_\alpha (at^k)$, with $a\in\mathbb C$. Furthermore, in the twinned buildings $\SHI^+_x$ and $\SHI^-_x$, elements of the form $x_\alpha (at^k), \ a\in\mathbb C$ acts transitively on the apartments containing the half-apartment given by $D(\alpha,k)$. See \cite{GR14}, Section 4.1, for more details.

\subsection{The paths}
\label{suse:Paths}

The elements of $Y$, through the identification $Y=N.0$, are called vertices of type $0$ in $\A$.

We consider piecewise linear continuous paths
$\pi:[0,1]\rightarrow \mathbb A$ such that each (existing) tangent vector $\pi'(t)$
belongs to an orbit $W^v.\lambda$ for some $\lambda\in {\overline{C^v_f}}$, and such that the endings of the path $\pi(0)$ and $\pi(1)$ are vertices of type $0$. Such a path is called a {\it $\lambda-$path}; it is
increasing with respect to the preorder relation $\leq$ on $\mathbb A$.

For any $t\neq 0$ (resp. $t \neq1$), we let $\pi'_-(t)$ (resp. $\pi'_+(t)$) denote the derivative of $\pi$ at $t$ from the left (resp. from the right). These derivatives are identified with elements of the twinned apartments $(\A_x^-,\A_x^+)$. For $x = \pi(t)$, we will also identify them with $-\pi(t -\varepsilon)$ (resp. $\pi(t +\varepsilon)$), for a small $\varepsilon >0$.


Let $\pi$ be a $\lambda-$path in $\mathbb A$. A wall $M$ is left positively by $\bar\pi$ with respect to $w\cdot \g S_{-\infty}$ if there exists $t\in [0,1]$, such that $\bar\pi(t) \in M$ and $M$ separates $\pi'_-(t)$ and $w\cdot\g S_{-\infty}$. This condition means that there exists a $V-$translate of $-w\cdot C_f^v$ separated from $\pi'_-(t)$ by $M$.

For any $\beta\in\Phi^-$, consider 
\[
\hbox{pos}_\beta^w (\pi) = \#\{M = M_{\beta,k}, M \hbox{ is left positively by } \bar\pi \hbox{ with respect to }w\cdot\g S_{-\infty}\}.
\] 
Further, set 
\[
\hbox{ddim}^w (\pi) = \sum_{\beta\in\Phi^-} \hbox{pos}_\beta^w (\pi).
\]
 Thanks to this statistic on paths, we define a Hecke path of shape $\ql$ with respect to $w\g S_{-\infty}$ to be a path $\pi$ such that $\hbox{ddim}^w (\pi) \leqslant \rho (\lambda -\mu)$, for $\mu = \pi (1) - \pi(0)$.
Finally, we say that $\pi$ is a LS path of shape $\ql$ with respect to $w\g S_{-\infty}$ if there is equality: $ \hbox{ddim}^w (\pi) = \rho (\lambda -\mu)$, for $\mu = \pi (1) - \pi(0)$. 

As shown in \cite{GR08}, Section 5.3, this definition of LS paths is equivalent to the Littelmann's one for $w=$ id, and we recover the classical notion of the LS paths.

\section{Sections}

In this section, we translate the definitions of~\cite[\S 5]{E10} into the world of paths.

\begin{defi}

Let $\pi$ be an LS path of shape $\lambda$ and let us fix a real positive root $\alpha$, not necessarily simple. An interval $[t_i,t_j]\subset [0,1]$ is 
\begin{itemize}
\item an $\alpha$-zero section at $m\in\mathbb Z$ if
\[
\alpha(\pi(t_i)) = m = \alpha(\pi(t_j)),\quad \forall t\in ]t_i,t_j[, \ \alpha(\pi(t)) = m;
\]
\item an $\alpha$-stable section at $m\in\mathbb Z$ if 
\[
\alpha(\pi(t_i)) = m = \alpha(\pi(t_j)),\quad \forall t\in ]t_i,t_j[, \ \alpha(\pi(t)) > m
\] 
and $[t_i,t_j]$ is maximal with respect to the existence of such an $m$;
\item an $\alpha$-directed section at $m\in\mathbb Z$ if 
\[
\alpha(\pi(t_i)) = m, \quad \alpha(\pi(t_j)) = m+1\quad\hbox{and} \quad \forall t\in ]t_i,t_j[, \ m<\alpha(\pi(t)) < m + 1;
\]
\item an $-\alpha$-directed section at $m\in\mathbb Z$ if 
\[
\alpha(\pi(t_i)) = m, \quad \alpha(\pi(t_j)) = m-1\quad\hbox{and} \quad \forall t\in ]t_i,t_j[, \ m-1<\alpha(\pi(t)) < m.
\]
\end{itemize}
\end{defi}

\medskip



\begin{prop}\label{partsec}
To any LS path $\pi$ and any real positive root $\alpha$, there exists a unique partition $0<t_1 <\cdots < t_\ell < 1$ of $[0,1]$ such that each interval $[t_{i},t_{i+1}]$ is either an $\alpha$-zero, $\alpha$-stable, $\alpha$-directed or $-\alpha$-directed section.
\end{prop}

\begin{proof}
Indeed, start from $0 = \pi(0)$, and pick the first $t_1\in[0,1]$, such that the positive tangent vector $\alpha(\pi'_+(t_1)) \ne 0$. If such a $t$ does not exist, the whole path lies in the wall $\ker(\alpha)$ and is an $\alpha$-zero section. If $\alpha(\pi'_+(t_1)) > 0$ and if there exists a $t_2>t_1$ such that $\alpha(\pi(t_2)) = 0$, then for the smallest one, we have a $\alpha$-stable section. If the path does not come back to the level $\alpha = 0$, there must be a $t_2$ such that $\alpha(\pi(t_2)) = 1$, so for the smallest one, we have a $\alpha$-directed section.

If $\alpha(\pi'_+(t_1)) < 0$, let $t_2>t_1$ be the smallest possible such that $\alpha(\pi(t_2)) = -1$, then $[t_1,t_2]$ is a $-\alpha$-directed section. Then one start all over again with $t_2$. 
\end{proof}

\begin{rema}
Note that a $(-\alpha)$-directed section can not occur after an $\alpha$-directed one.
\end{rema}

\medskip

Sections are compatible with the Littelmann's operators. Let $Q = \min\{\alpha(\pi(t)), t\in[0,1]\}$, as $\pi$ is an LS path, this minimum is an integer. Now fix $q = \min\{t\in [0,1], \alpha(\pi(t)) = Q \}$ and $p = \max\{t\in [0,1], \alpha(\pi(t)) = Q \}$.

If $Q = 0$ one set $e_\alpha \pi = 0$, whereas if $Q\leqslant -1$, let $y\in[0,1]$ such that  $\alpha(\pi(y)) = Q+1$ and for $y<t<q$, $Q<\alpha(\pi(t))<Q+1$. Then one cut the path $\pi$ into three pieces: $\pi = \pi_1 * \pi_2 *\pi_3$, with $\pi_1(t) = \pi(ty)$, $\pi_2 (t) = \pi(y + t(q-y)) - \pi(y)$ and $\pi_3 (t) = \pi(q-t(1-q)) - \pi(q)$. And $e_\alpha \pi = \pi_1 * s_\alpha(\pi_2) *\pi_3$

For $f_\alpha$, the definition is analogous. If $\alpha(\pi(1)) - Q = 0$, one set $f_\alpha \pi = 0$, but if $\alpha(\pi(1)) - Q\geqslant 1$, let $x\in[p,1]$ such that  $\alpha(\pi(x)) = Q+1$ and for $p<t<x$, $Q<\alpha(\pi(t))<Q+1$. Then one cut the path $\pi$ into three pieces: $\pi = \pi_1 * \pi_2 *\pi_3$, with $\pi_1(t) = \pi(tp)$, $\pi_2 (t) = \pi(p + t(x-p)) - \pi(p)$ and $\pi_3 (t) = \pi(x-t(1-x)) - \pi(x)$. And $f_\alpha \pi = \pi_1 * s_\alpha(\pi_2) *\pi_3$.

By definition, $[y,q]$ is a $-\alpha$-directed section for $\pi$ and an $\alpha$-directed one for $e_\alpha \pi$. Likewise, $[p,x]$ is a $\alpha$-directed section for $\pi$ and an $-\alpha$-directed one for $f_\alpha \pi$. So the partition of $\pi$ is preserved by the operators.

\medskip

Next let define the flip $(\pi)_{-\alpha}$ of $\pi$ with respect to $\alpha$. Cut $\pi$ according to the partition: $\pi = \pi_0*\pi_1 * \cdots * \pi_\ell$, where, for $0\leqslant i\leqslant \ell$, $\pi_i (t) = \pi(t_i - t(t_{i+1} - t_i)) - \pi(t_i)$. Then $(\pi)_{-\alpha}$ is obtained by concatenating the $\pi_i$ if $[t_i,t_{i+1}]$ is not stable and the $s_\alpha(\pi_i)$ if $[t_i,t_{i+1}]$ is stable. The condition of being LS is not preserved by this operation. 

\begin{prop}\label{propsec}
For any LS path $\pi$ and any real positive root $\alpha$, $(e_\alpha^{\max} \pi)_{-\alpha} = s_\alpha f_\alpha^{\max} \pi$.
\end{prop}

\begin{proof}
The operation $f_\alpha^{\max} \pi$ turns all the $\alpha$-directed sections into $-\alpha$ ones such that this path has only zero, stable and $-\alpha$-directed sections. Then by taking the reflection $s_\alpha$, we get back the $\alpha$-directed sections and the stables are also reflected. This is exactly what the operation $(e_\alpha^{\max} \pi)_{-\alpha}$ does.

\end{proof}

Now, let $w = s_{i_1}\cdots s_{i_n}$ a reduced expression of an element $w$ in the Weyl group. Let define two sequences of paths. First, let $\Theta_0 = \pi$ a LS path of shape $\lambda$ with $\pi(0)=0$. And, for $1\leqslant k\leqslant n$, set 
\[\Theta_k = f_{\alpha_{i_k}}^{\max}   (\Theta_{k-1}) = f_{\alpha_{i_k}}^{\max} f_{\alpha_{i_{k-1}}}^{\max} \cdots f_{\alpha_{i_1}}^{\max} (\pi).
\]
 Second, let $\Upsilon_0 (\pi) = \pi$ and set, for $1\leqslant k\leqslant n$, set 
\[
\Upsilon_{k} (\pi) = w(k) f_{\alpha_{i_k}}^{\max} (w(k-1)^{-1} \Upsilon_{k-1}(\pi)),
\]
 where $w(k) = s_{i_1}\cdots s_{i_k}$. A result of Kashiwara, in Section 7.2 of~\cite{Ka94}, on crystals implies that these paths do not depend on the reduced expression of $w$. So we set
\[
\Upsilon_w (\pi) = w\Theta_n = w f_w^{\max} (\pi),
\]
where $f_w^{\max} =  f_{\alpha_{i_n}}^{\max} f_{\alpha_{i_{n-1}}}^{\max} \cdots f_{\alpha_{i_1}}^{\max}$. One can note that $\Theta_k$ is LS for any $k$ and that $\Upsilon_w(\pi)$ is LS with respect to $w \g S_{-\infty}$. So if $s_\alpha$, for a simple root $\alpha$, is such that $\ell (ws_\alpha) > \ell (w)$ then $\Upsilon_{ws_\alpha} (\pi) = \Upsilon _{s_{w(\alpha)}} (\Upsilon_w (\pi))$.

\medskip
Consider now $\Upsilon'_w (\pi) = \Upsilon_w (\pi) + (\pi(1) - \Upsilon_w (\pi)(1))$, for any $w\in W^v$, the path $\Upsilon_{ws_\alpha} (\pi)$ translated by the vector $\pi(1) - \Upsilon_w (\pi)(1)$.

\begin{prop}\label{propConsecDiff}
Let $w\in W^v$ and let $w = s_{i_1}\cdots s_{i_n}$ be a reduced expression. Set $\beta_n^\vee = w(n-1) (\alpha_{i_n}^\vee)$. Then 
\[
\Upsilon'_{ws_{i_n}} (\pi) (0) - \Upsilon'_w (\pi)(0) = \varepsilon_{\alpha_{i_n}}(\Theta_{n-1})\beta_n^\vee,
\]
 where $\varepsilon_{\alpha_{i_n}}$ applied to an LS path is the number of times one can apply the root operator $e_{\alpha_{i_n}}$ to that path.
\end{prop}
\begin{proof}
Let us compute : 
\[
\Upsilon'_{ws_{i_n}} (\pi) (0) - \Upsilon'_w (\pi)(0) = \Upsilon_{ws_{i_n}} (\pi)(0) + (\pi(1) - \Upsilon_{ws_{i_n}} (\pi)(1)) - \Upsilon_w (\pi) (0) - (\pi(1) - \Upsilon_w (\pi)(1)). 
\]
 But $\Upsilon_{ws_{i_n}} (\pi)(0) = \Upsilon_w (\pi) (0) = 0$. So we are back to 
\[
\Upsilon'_{ws_{i_n}} (\pi) (0) - \Upsilon'_w (\pi)(0) = \Upsilon_w (\pi)(1) - \Upsilon_{ws_{i_n}} (\pi)(1).
\]
 If we multiply this by $w(n-1)^{-1}$, we get : 
\begin{align*}
w(n-1)^{-1} \big (\Upsilon'_{ws_{i_n}} (\pi) (0) - \Upsilon'_w (\pi)(0) \big ) & =  (w(n-1)^{-1}\Upsilon_w (\pi))(1) - (w(n-1)^{-1} \Upsilon_{ws_{i_n}} (\pi)) (1)\\
 & =  (s_{i_n}\Theta_n)(1) - \Theta_{n-1}(1)\\
 & =  (s_{i_n} f^{\max}_{\alpha_{i_n}}\Theta_{n-1})(1) - \Theta_{n-1}(1)\\
 & =  (e^{\max}_{\alpha_{i_n}}\Theta_{n-1})_{-\alpha_{i_n}}(1) - \Theta_{n-1}(1)\\
 & =  (e^{\max}_{\alpha_{i_n}}\Theta_{n-1})(1) - \Theta_{n-1}(1)\\
 & =  \varepsilon_{\alpha_{i_n}}(\Theta_{n-1})\alpha^\vee_{i_n}.
\end{align*}
\end{proof}

Now, we consider the crystal $\mathcal B(\infty)$ associated to the Langlands dual $G^\vee$ and identify it with the union of all the crystals $\mathcal B_\lambda$, where $\mathcal B_\lambda$ is the set of all the LS paths of shape $\lambda$ starting at $0$. Recall  that in~\cite{BKT14} are defined affine Mirkovi\'c-Vilonen polytopes $P_b$ attached to any crystal element $b\in\mathcal B(\infty)$. Thanks to~\cite{BGK12}, Proposition~\ref{propConsecDiff} implies the following analog of~\cite[Proposition 6.7]{E10}.

\begin{prop}\label{propBottomPol}
The polygonal line given by the set of coweights $\{\Upsilon'_w (\pi)(0)\mid w\in W^v\}$ is the bottom part of the affine MV polytope $P_b$, if $b$ corresponds to $\pi$.
\end{prop}

Let us give here some results on the paths $\Upsilon'_w (\pi)$ for later use.

\begin{lemm}\label{calcm}
Let $m = \min_{t\in[0,1]} \Big ( w(\alpha) \big ( \Upsilon'_w(\pi)(t)\big )\Big )$, then we have $m=w(\alpha) (\mu - \Upsilon_w(\pi) (1)) - \varepsilon_\alpha (\Theta_w)$.
\end{lemm}

\begin{proof}
Indeed,
\begin{align*}
\min\{w(\alpha) (\Upsilon'(\pi)(t)), t\in[0,1]\} &=  \min\{w(\alpha) (\Upsilon_w(t) + (\mu - \Upsilon_w (1)), t\in[0,1]\}\\
 &=  w(\alpha) (\mu - \Upsilon_w(\pi)  (1)) + \min\{w(\alpha) (\Upsilon_w(\pi) (t), t\in[0,1]\}\\
 &=  w(\alpha) (\mu - \Upsilon_w(\pi)  (1)) + \min\{w(\alpha) (w\Theta_w(t), t\in[0,1]\}\\
 &=  w(\alpha) (\mu - \Upsilon_w (\pi) (1)) + \min\{\alpha (\Theta_w(t)), t\in[0,1]\}\\
 &=  w(\alpha) (\mu - \Upsilon_w(\pi)  (1)) - \varepsilon_\alpha (\Theta_w)\\
\end{align*}
as expected.
\end{proof}

\begin{lemm}\label{calcreflection}
Let $q = \min\{ t\in[0,1]\mid w(\alpha)\big ( \Upsilon'_w(\pi))(t)\big ) = m \}$. Then $\Upsilon'_{ws_\alpha}(\pi)(t)=s_{-w(\qa),m}(\Upsilon'_w(\pi)(t))$, where $s_{-w(\qa),m}$ is the reflection along the wall $M(-w(\alpha), m)$ in $\mathbb A$.
\end{lemm}

\begin{proof}
One has
\begin{align*}
\Upsilon'_{ws_\alpha}(\pi)(t)&=s_{-w(\qa),m}(\Upsilon'_w(\pi)(t))\\
&=\Upsilon'_w (\pi)(t) - w(\alpha)\Upsilon'_w (\pi)(t)w(\alpha)^\vee + m w(\alpha)^\vee\\
&\overset{\ref{calcm}}{=}\Upsilon'_w (\pi)(t) - w(\alpha)\Upsilon'_w (\pi)(t)w(\alpha)^\vee + (w(\alpha) (\mu - \Upsilon_w (\pi) (1)) - \varepsilon_\alpha (\Theta_w)) w(\alpha)^\vee\\
&\overset{\ref{propConsecDiff}}{=}\Upsilon'_w (\pi)(t) - w(\alpha)\Upsilon'_w (\pi)(t)w(\alpha)^\vee + w(\alpha) (\mu - \Upsilon_w (\pi) (1))w(\alpha)^\vee\\&\qquad\qquad\qquad + \Upsilon_w(\pi)(1)-\Upsilon_{ws_\qa}(\pi)(1).
\end{align*}
Note that \begin{align*}
\Upsilon_{ws_\alpha}(\pi)(t)&=\Upsilon_{ws_\alpha}(\pi)(t)\\
&=s_{w(\qa)}f_{w(\qa)}^{\max}(\Upsilon_w(\pi)(t))\\
&\overset{\ref{propsec}}{=}e_{w(\qa)}^{\max}(\Upsilon_w(\pi)(t))_{-w(\qa)}\\
&=s_{-w(\qa)}(\Upsilon_w(\pi)(t)),
\end{align*}
which amounts to \begin{align*}
\Upsilon'_{ws_\alpha}(\pi)(t)&=\mu-\Upsilon_{ws_\qa}(\pi)(1)+s_{-w(\qa)}(\Upsilon'_w(\pi)(t)-\mu+\Upsilon_w(\pi)(1))\\
&=\mu-\Upsilon_{ws_\qa}(\pi)(1)+\Upsilon'_w(\pi)(t)-\mu\\
&\qquad\qquad+\Upsilon_w(\pi)(1)-w(\qa)(\Upsilon'_w(\pi)(t)-\mu+\Upsilon_w(\pi)(1))w(\qa)^\vee
\end{align*}
which is exactly what we want.

\end{proof}

\section{Retractions}
\label{se:Retractions}

\subsection{Definitions}
\label{suse:Definitions}

Let $\g S$ be any sector germ. For any $x\in \SHI^+$ there is an apartment containing $x$ and $\g S$ \cite[5.1]{R11} and this apartment is conjugated to $\A$ by an element of $G$ fixing $\g S$.
    So, by the usual arguments, using \cite[4.4]{GR08}, we can define the retraction $\qr_{\g S}$ of $\SHI$ onto the apartment $\A$ with center the sector germ $\g S$. 

For any such retraction $\qr_{\g S}$, the image of any segment $[x,y]$ with $x\leq y$ is a $\ql-$path, for some $\lambda\in \overline{C^v_f}$. 
Actually, the image by $\qr_{\g S}$ of any segment $[x,y]$ with $x\leq y$ and $x, y$ in $Y$ is a Hecke path of shape $\ql$ with respect to $\g S$. 

\begin{lemm}
\label{le:RetractSectorGerm}
Let $\qr_{\g S}$ be the retraction associated to the sector germ $\g S = w.\g S_{-\infty}$, for $w\in W^v$. 
Let $\gamma_w : \mathbb C^\times \to G$ be a regular one-parameter subgroup such that $w^{-1}.\gamma_w$ is antidominant. Let $x$ be a point in $\SHI$, then $\qr_{\g S}(x) = \lim_{s\to 0}\gamma_w (s). x$. 
\end{lemm}
\begin{proof}
There exists a representative $\mathfrak s$ of the sector germ $\g S$ in the standard apartment $\mathbb A$ such that $x$ and this sector are in a same apartment $A$. 
In terms of group, this means that there exist an element $g$ of the group $G$ and a point $z$ of $\mathbb A$ such that $A = g.\mathbb A$ and $x = g.z$. But since the action of $g$ stabilizes $\mathfrak s$, $g$ is an element in $w.U^-$. 

On one side, by definition of the retraction $\qr_{\g S}$, $z = \qr_{\g S} (x)$. On the other side, $\lim_{s\to 0} \gamma (s).x = \lim_{s\to 0} \gamma (s).g.z = z$, as $g$ is in $w.U^-$.  
\end{proof}

Let $\alpha = w(\alpha_i)$ be a positive root, with $w\in W^v$ and $\alpha_i$ a simple root. The masure $\SHI$ contains the extended tree $\SHT^w$ associated to $(\mathbb A, \alpha)$ that was defined in \cite{GR14} under the name $\SHI(M_\infty)$. 
Its standard apartment is $\A$ as affine space, the standard apartment of the masure, but with only walls the walls directed by $\ker \qa$.
There, it is also proven that the retraction $\rho_{w\g S_{-\infty}}$ factorizes through $\SHT^w$ and equals the composition \[
\rho_{w\g S_{-\infty}}:\SHI\stackrel{\rho^w_1}{\to}\SHT^w\stackrel{\rho^{-,w}_2}{\to}\mathbb A,\]
 where $\rho^w_1$ is the parabolic retraction defined in 5.6 of \cite{GR14} and $\rho^{\pm,w}_2$ is the retraction with center the end $\pm \infty_{\SHT^w}$, \ie the class of half-apartments in $\SHT^w$ containing $w\g S_{\pm\infty}$.

The parabolic retraction $\rho^w_1$ can also be defined in terms of limit of one-parameter subgroup. Actually, as defined in 5.6 of \cite{GR14}, $\rho^w_1$ is the retraction associated to the panel germ $\g F^w$, germ of the panel $w(-F^v(\{i\}))$. 

Adapting the proof of Lemma \ref{le:RetractSectorGerm}, we express the parabolic retraction in terms of the corresponding one-parameter subgroup.

\begin{lemm}
Let $\gamma^w_1 : \mathbb C^\times \to G$ be an antidominant one-parameter subgroup such that, for all $s$, $\alpha (\gamma^w_1 (s)) = 0$ and, for $j\ne i$, $\alpha_j (w^{-1}\gamma^w_1 (s)) < 0$. Let $x$ be a point in $\SHI$, then $\rho^w_{1}(x) = \lim_{s\to 0}\gamma^w_1 (s). x$. 
\end{lemm}

\subsection{Retraction in the extended tree}

Let $w\in W^v$, $\Upsilon$ be a LS path with respect to $w\g S_{-\infty}$ and $\gamma^w_1 : \mathbb C^\times \to G$ be an antidominant one-parameter subgroup as in the previous Lemma. Let $x$ be a point in the masure $\SHI$ given as 
\[
x = \prod_{j=1}^{\ell} x_{\beta_j}(a_j t^{n_j})\cdot a,
\] 
for $a = \Upsilon (\theta)\in\mathbb A$, $\theta\in [0,1]$, $a_j\in\mathbb C$, and the $\beta_j$'s are negative real roots.
Let us compute $\gamma^w_1(s)\cdot x$. We have:
\[
\gamma^w_1(s)\cdot x = \prod_{j=1}^{\ell} x_{\beta_j}(s^{\beta_j(\gamma^w_1(s))}a_j t^{n_j})\cdot a,
\] since $\gamma^w_1(s)\cdot a = a$, as $\gamma^w_1(s)\in T(\mathbb C)\subset T(\mathscr O)$. 
 But $\beta_j(\gamma^w_1(s))$ is positive except when $\beta_j = - w(\alpha)$, where in that case it is $0$. Hence 
\[
\rho^w_1 (x) = \lim_{s\to 0} \gamma^w_1\cdot x = x_{-w(\alpha)}(b_1 t^{m_1})\cdots x_{-w(\alpha)}(b_p t^{m_p})\cdot a,
\]
 where $b_i = a_{j_i}$, for $\beta_{j_i} = -w(\alpha)$ and among the integers $m_i$, denote by $m_q$ the minimum.
Hence, we can write
\[
\rho^w_1 (x) = x_{-w(\alpha)}(ct^{m_q}D) \cdot a,
\]
 where $D = 1 + d_1t + \cdots + d_mt^m$ is a polynomial in $t$,  for some $m \geqslant 0$, and $c = b_{q_1} + \cdots + b_{q_h}$, where $\{q_1,...,q_h\} = \{1\leqslant j\leqslant p,\ m_j = m_q\}$. 

\medskip
Now we assume that $a$ belongs to $D(w(\alpha), m_q)$ and $c\ne 0$. Further, using the $\mathsf {SL}_2$ relation $x_{-w(\alpha)}(A) = x_{w(\alpha)}(A^{-1})s_{w(\alpha)} (-A^{-1})^{-w(\alpha)^\vee} x_{w(\alpha)}(A^{-1})$, we get
\[
\rho^w_1 (x) = x_{w(\alpha)}((ct^{m_q} D)^{-1})s_{w(\alpha)} (-(ct^{m_q}D)^{-1})^{-w(\alpha)^\vee} x_{w(\alpha)}((ct^{m_q} D)^{-1})\cdot a. 
\]
 As $m_q\leqslant 0$ and the valuation of $(ct^{m_q} D)^{-1}$ is $-m_q$, the last term stabilizes $a$, so we get
\begin{align*}
\rho^w_1 (x)  & =  x_{w(\alpha)}((ct^{m_q} D)^{-1})s_{w(\alpha)} (-c)^{w(\alpha)^\vee}t^{m_qw(\alpha)^\vee}D^{w(\alpha)^\vee} \cdot a \\
  & =  x_{w(\alpha)}((ct^{m_q} D)^{-1})s_{w(\alpha)} t^{m_qw(\alpha)^\vee} \cdot a \\
 & =  x_{w(\alpha)}((ct^{m_q} D)^{-1})s_{-w(\alpha),m_q}  (a).
\end{align*} 
The first step is the fact that $D^{\alpha^\vee}$ and $(-c)^{\alpha^\vee}$ stabilize  $a$ and then $s_w(\alpha)t^{m_qw(\alpha)^\vee}\cdot a =  s_{-w(\alpha),m_q}  (a)$, for any $a\in\mathbb A$. In conclusion, from the two writings of $\rho^w_1 (x)$  we see that $\rho_2^-(\rho^w_1 (x)) = a$ and $\rho_2^+(\rho_1 (x)) = s_{-w(\alpha),m_q}  (a)$. 


\subsection{Retractions of a segment}

Let $\pi$ be an LS path of shape $\lambda$ (with respect to any sector-germ $w\g S_{-\infty}$), denote $\mu = \pi(1)$. Let $[x,\mu]$ be a segment in the masure such that $\rho_{\g S_{-\infty}} ([x,\mu]) = \pi$. In particular, this implies that 
\[
x = \prod_{j=1}^{\ell} x_{\beta_j}(a_j t^{n_j})\cdot [t^{\pi(0)}],
\]
 where $1\geqslant r_1\geqslant\cdots \geqslant r_\ell\geqslant 0$ are the times where the reverse path $\bar\pi$ leave a wall $M(\beta_j,n_j)$ in the positive direction, the $\beta_j$'s being distinct negative roots (several roots can have the same time). Further, as shown in \cite[6.2]{GR08}, the set of segments $[x,\mu]$ in $\SHI$ that retract onto $\pi$ is nonempty and is parameterised by $\ell = \hbox{ddim}(\pi)$ parameters. Further, the set of parameters $(a_j)_{1\leqslant j\leqslant \ell}$ is a finite product of $\mathbb C$ and $\mathbb C^*$. Let us denote this set by $P_{w}(\qp,\pi(1))$.

\medskip
In this section, we want to show the following theorem. 

\begin{theo}
\label{retractdense}
Let $\pi$ be an LS path of shape $\lambda$, with $\pi(0)=0$. For any $w\in W^v$, there exists an dense subset of the parameters $\mathscr O_w\subset P_{w}(\qp,\pi(1))$, such that for all $x\in\mathscr O_w$,
\[
\rho_{w.\g S_{-\infty}} ([x,\mu]) = \Upsilon'_w (\pi),
\]
 the translation of $\Upsilon_w(\pi) = w f_w^{\max} (\pi)$ so that the path ends at $\mu = \pi(1)$.
\end{theo}

The proof proceeds by induction on the length of $w$. If $w = 1$, we are done as, by assumption, $\rho_{\g S_{-\infty}} ([x,\mu]) = \pi$.

Consider a simple root $\alpha$ and an element $w\in W^v$ such that $\ell (ws_\alpha) > \ell (w)$, recall that $\Upsilon_{ws_\alpha} (\pi) = \Upsilon _{s_{w(\alpha)}} \Upsilon_w (\pi)$. Assume that $\rho_{w\g S_{-\infty}} ([x,\mu]) = \Upsilon'_w (\pi)$, for $x$ in a dense subset $\mathscr O_w\subset P_{w}(\qp,\pi(1))$. We want to show that $\rho_{ws_\alpha \g S_{-\infty}} ([x,\mu]) = \Upsilon'_{ws_\alpha}(\pi) $, for $x$ in a dense subset of the parameters.

Let $\rho_1^w : \SHI\to \SHT^w$ be the parabolic retraction associated to the panel germ $\g F^w$ separating $ws_\alpha \g S_{-\infty}$ and $w \g S_{-\infty}$, where $\SHT^w$ is the extended tree associated to $\g F^w$. Let $\rho_2^{\pm,w} : \SHT^w\to\mathbb A$ be the retraction with center $\pm \infty_{\SHT^w}$ from the extended tree to the standard apartment. So we have, for any $z\in\SHI$,
$\rho_{w\g S_{-\infty}} (z) = \rho_2^{-,w}\circ \rho_1^w (z)$ and $\rho_{ws_\alpha \g S_{-\infty}} (z) =  \rho_2^{+,w}\circ \rho_1^w (z)$. 

We consider the path $\eta = \rho_1^w([x,\mu])$ in the extended tree $\SHT^w$, for $x\in \mathscr O_w$. The reverse path $\overline{\eta}$ starts at $\mu$ in $\mathbb A$ and then changes apartment only at walls parallel to $\ker\alpha$, via elements of the group like $x_{-w(\alpha)} (at^m)$, with $a$ a complex number and $m$ an integer. Since we assume that $\rho_{w\g S_{-\infty}} ([x,\mu]) = \Upsilon'_w (\pi)$, for $x\in\mathscr O_w$, we have $ \rho_2^{-,w} (\eta) = \Upsilon'_w (\pi)$.

\medskip
Let $[r,s]\subset [0,1]$ be a $w(\alpha)$-stable section of $\Upsilon'_w (\pi)$ at a wall $M(-w(\alpha),k)$, for some $k\in\mathbb Z$. 
Suppose that $\eta(s)\in M(-w(\alpha),k)\subset\mathbb A$.
Then $\eta(r)$ also belongs to $M(-w(\alpha),k)$ and  $\eta(s-\varepsilon)$, for $\varepsilon >0$, is in an apartment obtained by applying to $\mathbb A$ an element of the group of the form $x_{-w(\alpha)}(ct^k)$. 

\begin{lemm}\label{retstsec}
Suppose that $c\ne 0$, then $(\rho_2^{+,w}\circ\eta) [r,s]$ is the image of the $w(\alpha)$-stable section of $\Upsilon'_w (\pi)$ by the reflection $s_{-w(\alpha),k}$ along the wall $M(-w(\alpha),k)$.
\end{lemm}
\begin{proof}
By the induction assumption, the path $\rho_2^{-,w}\circ\eta = \Upsilon'_w(\pi)$, for $x$ in $\mathscr O_w$. Because the parameter $c$ is not zero, the path $\overline{\eta}$ leaves the standard apartment at $\eta(s-\varepsilon)$ and will be back at $\eta(r)$. Now the $w(\alpha)$-stable section $[r,s]$ of $\Upsilon'_w (\pi)$ may cross several walls parallel to $M(-w(\alpha),k)$, so that the path $\eta ([r,s])$ may lie in an apartment given by an element of the form $x_{-w(\alpha)}(ct^k +c_{k+1}t^{k+1}+\cdots + c_mt^m )$, for some integer $m>k$. 

On one side, the retraction $\rho_2^{-,w}$ sends $\eta([r,s])$ to the $w(\alpha)$-stable section of $\Upsilon'_w (\pi)$. But on the other side, performing the $\mathsf {SL}_2$ change of variables associated to the extended tree $\SHT^w$ to compute the retraction $\rho_2^{+,w}$, as in the previous section, we see that it sends $\eta([r,s])$ to the reflection of the $w(\alpha)$-stable section by $s_{-w(\alpha),k}$.
\end{proof}

Recall that we fixed
\[
m = \min_{t\in[0,1]} \Big ( w(\alpha) \big ( \Upsilon'_w(\pi)(t)\big )\Big )
\]
 and let 
\[
q = \min\{ t\in[0,1], \quad w(\alpha)\big ( \Upsilon'_w(\pi))(t)\big ) = m \}.
\]
 We know that $\rho_2^{-,w} (\eta(q))$ belongs to the wall $M(-w(\alpha),m)$.

\begin{lemm}\label{retsym+-}
Suppose that $\eta(q)\in M(-w(\alpha),m)$ and that $\eta(q-\varepsilon)$, $\varepsilon >0$, is in an apartment obtained by applying to $\mathbb A$ an element of the group of the form $x_{-w(\alpha)}(dt^{m'})$, with $d\ne 0$. Then $\rho_2^{+,w}(z) = s_{-w(\alpha),m}(\rho_2^{-,w} (z))$, for any $z = \eta(t)$, with $0\leqslant t\leqslant q$.

\end{lemm}
\begin{proof}
Thanks to the assumptions, the apartment containing $\eta(q-\varepsilon)$ shares an half-apartment with $\mathbb A$ along the wall $M(-w(\alpha),m')$. So computing in the extended tree, using the same techniques as in the proof of Lemma \ref{retstsec}, we see that 
$\rho_2^{+,w} (z)  =  s_{-w(\alpha), m'} (\rho_2^{-,w} (z)) $.
\end{proof}

We are ready to conclude our induction step. We want to prove that $\rho_{ws_\alpha \g S_{-\infty}} ([x,\mu]) = \Upsilon'_{ws_\alpha}(\pi) $ for $x$ in a dense subset of the set of parameters, assuming $\rho_{w\g S_{-\infty}} ([x,\mu]) = \Upsilon'_w (\pi)$, for $x$ in $\mathscr O_w$. 

Both retractions factorise through the path $\eta = \rho_1^w([x,\mu])$ contained in the extended tree $\SHT^w$. Now, we compute the retraction $\rho_2^{+,w} (\eta)$ stepwise by retracting some $w(\alpha)$-section at a time and starting backwards from $\mu = \pi(1) = \eta (1)\in\mathbb A$.

\medskip
We know that $\eta (1 - \varepsilon)$ is also in $\mathbb A$. Let $r_1\in [0,1]$ the biggest real such that $\eta(r_1)\in\mathbb A$ and $\eta (r_1-\varepsilon)\not\in\mathbb A$. Then $\eta(r_1)$ belongs to a wall $M(-w(\alpha),k)$ and $\eta (r_1-\varepsilon)$ is in an apartment given by acting by $x_{-w(\alpha)}(ct^k)$ on $\mathbb A$. If there exists $s\in [0,r_1]$ such that $\eta(s)\in M(-w(\alpha),k)$, then the path $\Upsilon'_w(\pi)$ admits a $w(\alpha)$-stable section at that wall. And we are in the position of applying Lemma \ref{retstsec}.

But there might be several $w(\alpha)$-stable sections at that wall. Let $c_1, ... , c_{j_1}$ be the parameters associated to those stable sections at times $s_1,..., s_{j_1}$ respectively. And since we want all the stable sections to be flipped along their walls, we have to impose that $c_1 + c_2 + \cdots + c_{j} \ne 0$, for all $1\leqslant j\leqslant j_1$. 
So, by applying Lemma \ref{retstsec}, we retract those stable sections to obtain a path $\eta_1$ such that $\eta_1 (t) = \Upsilon'_{ws_\alpha}(\pi)(t)$ for $t\in [s_{j_1},1]$. 

We start the procedure again with $\eta_1$, until we reach the next wall with some stable sections. 
We repeat this procedure until we reach the wall $M(-w(\alpha),m)$. Again taking into account the fact that stable sections may exist at times between $q$ and the last $r_k$ on the wall $M(-w(\alpha),m)$, we have to impose some analogous conditions on the parameters. 

\medskip

Finally, we are left with $0\le t\le q$. Thanks to~\ref{retsym+-}, this case is dealt with using~\ref{calcreflection}. And the conditions on the parameters define the dense subset $\mathscr O_{ws_\alpha}$. Note that the parameters involved in this inductive step were not involved in the computation before as they are attached to the negative root $-w(\alpha)$.

\section{$\widehat{\mathfrak{sl}}_2$ MV polytopes}

We know thanks to~\cite{MT14} that the description of rank $2$ affine MV polytopes given in~\cite{BDKT12} matches the original definition of~\cite{BKT14}. 
As we will only use the definitions and results of~\cite{BDKT12}, we will use the formalism used therein and start with a recollection of its content. The only adaptation needed is the one induced by working with $\mathcal B(\infty)$ instead of $\mathcal B(-\infty)$.

\subsection{Recollection of Baumann--Dunlap--Kamnitzer--Tingley}

Denote by $\qa_0$ and $\qa_1$ the simple roots for $\widehat{\mathfrak{sl}}_2$, and $\delta=\qa_0+\qa_1$ the primitive imaginary root.
Consider fundamental weights $\omega_0$ and $\omega_1$ satisfying $(\qa_i,\omega_j)=\delta_{i,j}$.
We call a \emph{Lusztig datum} a family $(a_k,\lambda_k,a^k)_{k\in \N_{>0}}$ of nonnegative integers, with finite support, such that $\lambda_1\ge\lambda_2\ge\dots$ defines a partition $\lambda$ of size $|\lambda|:=\sum_{k\ge1}\lambda_k$.
To a pair of Lusztig data $(a_k,\lambda_k,a^k)_{k\in \N_{>0}}$ and $(\bar a_k,\bar\lambda_k,\bar a^k)_{k\in \N_{>0}}$ and a weight $\mu_0=\bar\mu_0$ we associate vertices $\mu_k,\bar\mu_k,\mu^k,\bar\mu^k$, $k\in\N$, defined by\begin{align*}
\mu_k-\mu_{k-1}&=a_k(\qa_1+(k-1)\delta)\\
\bar\mu_k-\bar\mu_{k-1}&=\bar a_k(\qa_0+(k-1)\delta)\\
\mu_\infty&=\lim_\infty\mu_k\\
\bar\mu_\infty&=\lim_\infty\bar\mu_k\\
\mu^\infty&=\lim_\infty\mu^k=\mu_\infty+|\lambda|\delta\\
\bar\mu^\infty&=\lim_\infty\bar\mu^k=\bar\mu_\infty+|\bar\lambda|\delta\\
\mu^{k-1}-\mu^k&=a^k(\qa_0+(k-1)\delta)\\
\bar\mu^{k-1}-\bar\mu^{k}&=\bar a^k(\qa_1+(k-1)\delta).
\end{align*}
where the limits make sense thanks to the finite support condition. We obtain a polytope if $\bar\mu^0=\mu^0$, in which case we say that our initial Lusztig data have same weight. We call $(\bar a_k,\bar\lambda_k,\bar a^k)$ the left datum, and $(a_k,\lambda_k,a^k)$ the right one. We say that the polytope is decorated by the partitions $\bar\lambda$ and $\lambda$.

\begin{defi}\label{affmvdef}
Such a polytope is called an \textit{affine MV polytope} if \begin{enumerate}[label=(\roman*)]
\item\label{basseq}  for each $k \ge 2$, $( \bar\mu_{k}-\mu_{k-1}, \omega_1) \le 0$ and $(\mu_{k} -\bar\mu_{k-1}, \omega_0) \le 0$, with at least one of these being an equality;
\item\label{hauteq}   for each $k \ge 2$, $(\bar\mu^{k} -\mu^{k-1}, \omega_0) \ge 0$ and $(\mu^{k} -\bar\mu^{k-1}, \omega_1) \ge 0$, with at least one of these being an equality;
\item If $(\mu_\infty, \bar\mu_\infty)$ and $(\mu^\infty, \bar\mu^\infty)$ are parallel then $\lambda=\bar\lambda$. Otherwise, one is obtained from the other by removing a part of size $ (\mu_\infty - \bar\mu_\infty, \alpha_1)/2$;
\item  $ \lambda_1, \bar\lambda_1  \leq (\mu_\infty - \bar\mu_\infty, \alpha_1)/2$.
\end{enumerate}
We will denote by $MV$ the set of (affine) MV polytopes (up to translation).
\end{defi}
It is proved in~\cite[Theorem 3.11]{BDKT12} that for any given Lusztig datum $\mathbf a$, there is a unique MV polytope $P_{\mathbf a}$ whose right datum is $\mathbf a$. We reproduce the example given in \textit{loc.\!\! cit.}:

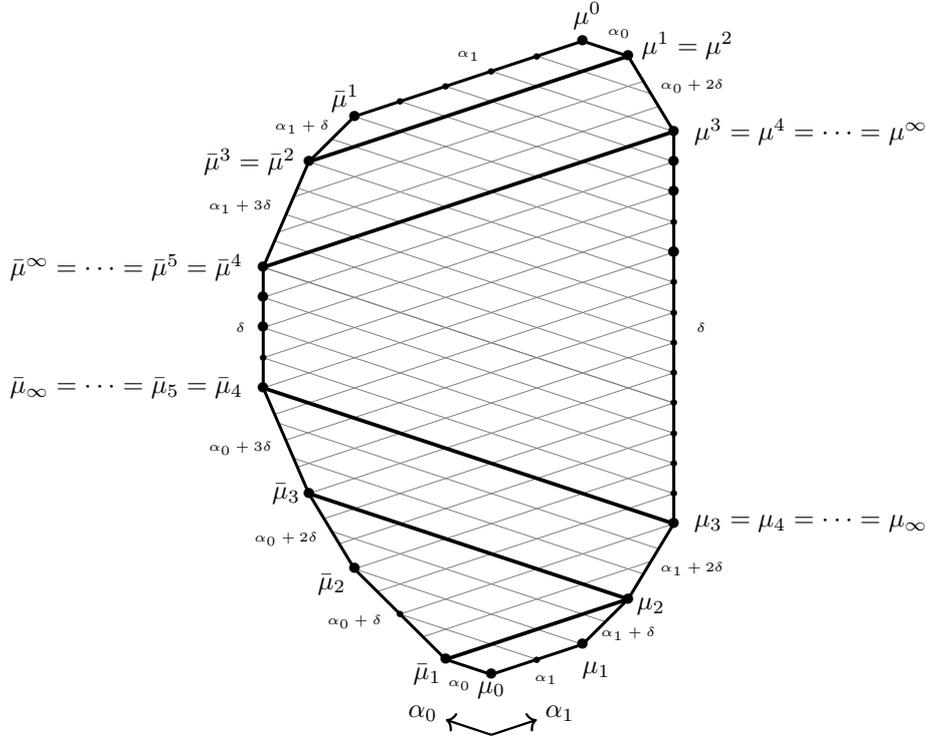
\raisebox{0.05\height}{\begin{tikzpicture}[yscale=0.2, xscale=0.6]

\draw [line width = 0.01cm, color=gray] (-1.5,2.5) -- (3.5, 7.5);
\draw [line width = 0.01cm, color=gray] (-2,4) -- (4, 10);
\draw [line width = 0.01cm, color=gray] (-2.5,5.5) -- (4, 12);
\draw [line width = 0.01cm, color=gray] (-3.3333,8.6666) -- (4, 16);
\draw [line width = 0.01cm, color=gray] (-3.6666,10.3333) -- (4, 18);
\draw [line width = 0.01cm, color=gray] (-4.25,13.75) -- (4, 22);
\draw [line width = 0.01cm, color=gray] (-4.5,15.55) -- (4, 24);
\draw [line width = 0.01cm, color=gray] (-4.75,17.25) -- (4, 26);
\draw [line width = 0.01cm, color=gray] (-5,19) -- (4, 28);
\draw [line width = 0.01cm, color=gray] (-5,21) -- (4, 30);
\draw [line width = 0.01cm, color=gray] (-5,23) -- (4, 32);
\draw [line width = 0.01cm, color=gray] (-5,25) -- (4, 34);
\draw [line width = 0.01cm, color=gray] (-4.6666,29.3333) -- (3.6666, 37.6666);
\draw [line width = 0.01cm, color=gray] (-4.3333,31.6666) -- (3.33333, 39.3333);

\draw [line width = 0.01cm, color=gray] (1,1) -- (-2,4);
\draw [line width = 0.01cm, color=gray]  (2,2)-- (-3,7);
\draw [line width = 0.01cm, color=gray] (2.5,3.5) -- (-3.5,9.5);
\draw [line width = 0.01cm, color=gray] (3.3333,6.6666) -- (-4.3333,14.3333);
\draw [line width = 0.01cm, color=gray] (3.6666,8.3333) -- (-4.6666,16.6666);
\draw [line width = 0.01cm, color=gray] (4,12) -- (-5,21);
\draw [line width = 0.01cm, color=gray] (4,14) -- (-5,23);
\draw [line width = 0.01cm, color=gray] (4,16) -- (-5,25);
\draw [line width = 0.01cm, color=gray] (4,18) -- (-5,27);

\draw [line width = 0.01cm, color=gray] (4,20) -- (-4.75,28.75);
\draw [line width = 0.01cm, color=gray] (4,22) -- (-4.5,30.5);
\draw [line width = 0.01cm, color=gray] (4,24) -- (-4.25,32.25);
\draw [line width = 0.01cm, color=gray] (4,28) -- (-3.5,35.5);

\draw [line width = 0.01cm, color=gray] (4,32) -- (-2,38);
\draw [line width = 0.01cm, color=gray] (4,34) -- (-1,39);
\draw [line width = 0.01cm, color=gray] (4,36) -- (0,40);
\draw [line width = 0.01cm, color=gray] (3.5,38.5) -- (1,41);

\draw [line width = 0.01cm, color=gray] (-3,7) -- (4, 14);
\draw [line width = 0.01cm, color=gray] (-4,12) -- (4, 20);

\draw [line width = 0.01cm, color=gray] (-3,37) -- (4, 30);
\draw [line width = 0.01cm, color=gray] (-4,34) -- (4, 26);
\draw [line width = 0.01cm, color=gray] (-5,27) -- (4, 18);

\draw (0,0) node {$\bullet$};

\draw (-1,1) node {$\bullet$};
\draw (-3,7) node {$\bullet$};
\draw (-4,12) node  {$\bullet$};
\draw (-5,19) node  {$\bullet$};
\draw (-5,23) node {$\bullet$};
\draw (-5,25) node {$\bullet$};
\draw  (-5,27) node {$\bullet$};
\draw (-4,34) node {$\bullet$};
\draw (-3,37) node {$\bullet$};
\draw  (2,42)  node {$\bullet$};

\draw (2,2)  node {$\bullet$};
\draw (3,5) node {$\bullet$};
\draw (4,10) node {$\bullet$};
\draw (4,28) node {$\bullet$};

\draw (4,28) node {$\bullet$};
\draw (4,32) node {$\bullet$};
\draw (4,34) node {$\bullet$};
\draw (4,36) node {$\bullet$};
\draw (3,41) node {$\bullet$};

\draw [line width = 0.04cm] (0,0)--(-1,1);
\draw  [line width = 0.04cm] (-1,1)--(-3,7);
\draw [line width = 0.04cm] (-3,7)--(-4,12);
\draw [line width = 0.04cm] (-4,12)--(-5,19);
\draw [line width = 0.04cm] (-5,19)--(-5,27);
\draw [line width = 0.04cm] (-5,27)--(-4,34);
\draw [line width = 0.04cm] (-4,34)--(-3,37);
\draw [line width = 0.04cm](-3,37)--(2,42);

\draw [line width = 0.04cm](0,0)--(2,2);
\draw [line width = 0.04cm] (2,2)--(3,5);
\draw [line width = 0.04cm] (3,5)--(4,10);
\draw [line width = 0.04cm] (4,10)--(4,36);
\draw [line width = 0.04cm](4,36)--(3,41);
\draw [line width = 0.04cm] (3,41)--(2,42);

\draw [line width = 0.05cm] (-1,1) -- (3,5);
\draw [line width = 0.05cm] (3,5) -- (-4,12);
\draw [line width = 0.05cm] (4,10) -- (-5,19);
\draw [line width = 0.05cm] (-5,27) -- (4,36);
\draw [line width = 0.05cm] (-4,34) -- (3,41);

\draw (1,1) node {\tiny $\bullet$};

\draw (4,12) node {\tiny $\bullet$};
\draw (4,14) node {\tiny $\bullet$};
\draw (4,16) node {\tiny $\bullet$};
\draw (4,18) node {\tiny $\bullet$};
\draw (4,20) node {\tiny $\bullet$};
\draw (4,22) node {\tiny $\bullet$};
\draw (4,24) node {\tiny $\bullet$};
\draw (4,26) node {\tiny $\bullet$};
\draw (4,30) node {\tiny $\bullet$};

\draw (-2,4) node {\tiny $\bullet$};
\draw (-5,21) node {\tiny $\bullet$};

\draw (-2,38) node {\tiny $\bullet$};
\draw (-1,39) node {\tiny $\bullet$};
\draw (0,40) node {\tiny $\bullet$};
\draw (1,41) node {\tiny $\bullet$};

\draw (1.2,-0.35) node {\tiny $\alpha_1$};
\draw (3,2.7) node {\tiny $\alpha_1+\delta$};
\draw (4.4,7) node {\tiny $\alpha_1+2\delta$};
\draw (4.6,23) node {\tiny $\delta$};
\draw (4.4,39) node {\tiny $\alpha_0+2\delta$};
\draw (2.8,42.5) node {\tiny $\alpha_0$};

\draw (-0.5,41) node {\tiny $\alpha_1$};
\draw (-4.15,36.3) node {\tiny $\alpha_1+\delta$};
\draw (-5.5,31) node {\tiny $\alpha_1+3\delta$};
\draw (-5.5,23) node {\tiny $\delta$};
\draw (-5.5,15) node {\tiny $\alpha_0+3\delta$};
\draw (-4.5,9) node {\tiny $\alpha_0+2\delta$};
\draw (-3,3.5) node {\tiny $\alpha_0+\delta$};
\draw (-0.7,-0.7) node {\tiny $\alpha_0$};

\draw(0,-1) node {$\mu_0 $};
\draw(2.3,0.3) node {$\mu_1$};

\draw(3.5,4.5) node {$\mu_2$};
\draw(7,10) node {$\mu_3= \mu_4 = \cdots = \mu_\infty$};
\draw(7,36) node {$\mu^3= \mu^4 = \cdots = \mu^\infty$};
\draw(4.3,41.7) node {$\mu^1= \mu^2$};

\draw(-1.4,0) node {$\bar \mu_1$};
\draw(-3.5,6) node {$\bar \mu_2$};
\draw(-4.5,12) node {$\bar \mu_3$};
\draw(-8,19) node {$\bar \mu_\infty = \cdots = \bar \mu_5= \bar \mu_4$};
\draw(-8,27) node {$\bar \mu^\infty = \cdots = \bar \mu^5= \bar \mu^4$};
\draw(-5.3,34) node {$\bar \mu^3= \bar \mu^2$};
\draw(-3.2,38.2) node {$\bar \mu^1$};
\draw(2.1,43.7) node {$\mu^0$};

\draw[line width = 0.03cm, ->] (0,-4)--(1,-3);
\draw[line width = 0.03cm, ->] (0,-4)--(-1,-3);
\draw (1.5,-2.6) node {$\alpha_1$};
\draw (-1.5,-2.6) node {$\alpha_0$};

\end{tikzpicture}}

Here, for instance, we have $a_1=2$, $a_2=a_3=1$, $a_k=0$ for $k\ge4$, $\lambda=(9,2,1^2)$, $a^k=0$ for $k\ge4$, $a^3=1$, $a^2=0$, $a^1=1$ for the right Lusztig datum. Bold diagonals correspond to a choice, for each $k\ge2$, of an \emph{active} diagonal, that is satisfying equality in~\ref{basseq} or~\ref{hauteq} in~\ref{affmvdef}. Such a choice form a \emph{complete system} of active diagonals. 

Thanks to~\cite[Remark 3.7]{BDKT12}, we have the following.

\begin{prop}\label{actdiag}
Any active diagonal cuts a MV polytope in two MV polytopes.
\end{prop}

 We use a different convention than in~\cite{BDKT12} and define a crystal structure on the set $MV$ of MV polytopes by setting 
 $f_0(P_{\mathbf a})=P_{f_0(\mathbf a)}$ where $f_0(\mathbf a)$ just adds $1$ to $\bar a_1$, and $f_1(P_{\mathbf a})=P_{f_1(\mathbf a)}$ where $f_0(\mathbf a)$ just adds $1$ to $a_1$ (here we simply use $f_i$ instead of $f_{\alpha_i}$ for $i\in\Z/2\Z$.
We finally recall that~\cite[Theorem 4.5]{BDKT12} proves that $MV$ realizes the crystal $\mathcal B(\infty)$.

\subsection{Link with retractions}\label{MVretract}

Note here that we have $G^\vee=G$.
Consider $b\in \mathcal B(\infty)$. We will denote by $P_b$ the corresponding MV polytope. Thanks to~\ref{propBottomPol}, the lower vertices of the  MV polytope associated to $b$ are given, up to a common translation, by the weights of\[
s_{i_1}\dots s_{i_k} f_{i_k}^{\max}\dots  f_{i_1}^{\max}(b)
\]
where $i_p\in\Z/2\Z$ and $i_{p+1}=i_p+1$. Precisely, we get $\overline \mu_k$ (resp.\ $\mu_k$) when $i_1=0$ (resp.\ $i_1=1$).
Thanks to~\ref{propBottomPol} and~\ref{retractdense}, this bottom part can hence be recovered using retractions on paths.

\begin{rema}
As a consequence, vertices $\overline \mu^k,\mu^k$ of the upper part of the polytope are given by the opposite of the weights of\[
s_{i_1}\dots s_{i_k} f_{i_k}^{\max}\dots  f_{i_1}^{\max}(b^*)
\]
where ${}^*$ denotes Kashiwara's involution. Indeed ${}^*$ simply transforms a MV polytope in its opposite, see~\cite[Definition 4.3 $\&$ Theorem 4.5]{BDKT12}
\end{rema}

\subsection{Top polytopes}

\begin{defi}
We call \emph{top polytope} an MV polytope $P$ such that $\lambda=\bar\lambda=0$ and $a_\bullet=\bar a_\bullet=0$ except $a_1$ or $\bar a_1$.
We denote by $MV^t$ the set of top polytopes. 
\end{defi}

Thanks to~\ref{actdiag}, the piece of a MV polytope $P$ above the diagonal $[\bar\mu^\infty,\mu^\infty]$ is itself a MV polytope, that we denote by $P^t$.
Denote by $b_0$ the highest weight element of $\mathcal B(\infty)$.

\begin{prop}\label{toponly}
Consider $b\in \mathcal B(\infty)$. We have $P_b\in MV^t$ if and only if \[
b=f_{i_1}^{k_1}\dots f_{i_t}^{k_t}(b_0)\]
where $i_{k+1}=i_k+1\in\Z/2\Z$ and $k_1>\dots>k_t$.
\end{prop}

\begin{proof} 
Assume that $\lambda=\bar\lambda=\bar a_\bullet=a_{\ge2}=0$ and $a_1=N>0$ so that $\epsilon_1(b)=N$. By induction on the weight (that is on the height $=(-,\omega_0+\omega_1)$), we need to prove that $\epsilon_0(e_1^N(b))=N'<N$ since clearly $P_{e_1^N(b)}\in MV^t$. Consider a complete system of active diagonals for $P_b$ and assume that it contains one directed by $\qa_0$. Thanks to~\cite[Proposition 4.10]{BDKT12}, we can assume that our polytope is the one contained between the lowest diagonal $N\alpha_1$ and the lower active diagonal $n\qa_0$ directed by $\qa_0$:\begin{equation}\label{polN}
\begin{tikzpicture}[scale=0.85,baseline=(current  bounding  box.center)]
\draw(0,0) node[left]{\small$\bar\mu^\infty=\dots=\bar\mu^{k_s+1}$};
\draw(0.5,2.3) node[left]{\small$\bar\mu^{k_s}=\dots=\bar\mu^{k_{s-1}+1}$};
\draw(1.1,4.6)node[left]{\small$\bar\mu^{k_{r+2}}=\dots=\bar\mu^{k_{r+1}+1}$};
\draw(2,6.6)node[left]{\small$\bar\mu^{k_{r+1}}=\dots=\bar\mu^{k_{r}+1}$};
\draw(4,8.2)node[above]{\small$\bar\mu^{k_r-1}$};
\draw(10,2) node[right]{\small$\mu^{k_s}=\dots=\mu^\infty$};
\draw(9.25,4.05) node[right]{\small$\mu^{k_{s-1}}=\dots=\mu^{k_s-1}$};
\draw(8.25,6.03) node[right]{\small$\mu^{k_{r+1}}=\dots=\mu^{k_{r+2}-1}$};
\draw(7,7.6) node[right]{~\small$\mu^{k_{r}}=\dots=\mu^{k_{r+1}-1}$};
\draw(0.5,2.3)--(0,0)--(10,2)node [midway,below]{\small$N\qa_1$};
\draw(0.5,2.3)(0,0)(10,2)(9.25,4.05)--(0.5,2.3);
\draw(0.5,2.3)--(0,0)node[midway,left]{\small$p_s(\qa_1+k_s\delta)$};
\draw(10,2)--(9.25,4.05)node[midway,right]{\small$p_s(\qa_0+(k_s-1)\delta)$};
\draw[dashed](0.5,2.3)--(1.1,4.6)  (8.25,6.03)--(9.25,4.05);
\draw (1.1,4.6)(2,6.6)--(7,7.6)(8.25,6.03)--(1.1,4.6);
\draw (1.1,4.6)--(2,6.6)node[midway,left]{\small$p_{r+1}(\qa_1+k_{r+1}\delta)$};
\draw  (7,7.6)--(8.25,6.03)node[midway,right]{~\small$p_{r+1}(\qa_0+(k_{r+1}-1)\delta)$};
\draw (7,7.6)--(4,8.2) node [midway,above]{\small$n\qa_0$};
\draw[dashed] (4,8.2)--(2,6.6)node [midway,above]{\small$v$};
\end{tikzpicture}
\end{equation}
where $v=g\alpha_1+h\delta$ for some $g>0$, $h\ge0$.
Using~\cite[Proposition 4.10]{BDKT12}, applying $e_1^N$ yields
\begin{equation}\label{polN'}
\begin{tikzpicture}[scale=0.9,baseline=(current  bounding  box.center)]
\draw(1,3.8) node[left]{\small$\bar\mu^\infty=\dots=\bar\mu^{k_s-1}$};
\draw(1.5,5.6) node[left]{\small$\bar\mu^{k_s-2}=\dots=\bar\mu^{k_{s-1}-1}$};
\draw(2.5,7.18)node[left]{\small$\bar\mu^{k_{r+2}-2}=\dots=\bar\mu^{k_{r+1}-1}$};
\draw(4,8.2)node[above]{\small$\bar\mu^{k_{r+1}-2}=\dots=\bar\mu^{k_{r}-1}$};
\draw(10,2) node[right]{\small$\mu^{k_s}=\dots=\mu^\infty$};
\draw(9.25,4.05) node[right]{\small$\mu^{k_{s-1}}=\dots=\mu^{k_s-1}$};
\draw(8.25,6.03) node[right]{\small$\mu^{k_{r+1}}=\dots=\mu^{k_{r+2}-1}$};
\draw(7,7.6) node[right]{~\small$\mu^{k_{r}}=\dots=\mu^{k_{r+1}-1}$};
\draw(10,2)--(9.25,4.05)node[midway,right]{\small$p_s(\qa_0+(k_s-1)\delta)$};
\draw[dashed](8.25,6.03)--(9.25,4.05);
\draw  (7,7.6)--(8.25,6.03)node[midway,right]{~\small$p_{r+1}(\qa_0+(k_{r+1}-1)\delta)$};
\draw (7,7.6)--(4,8.2) node [midway,above]{\small$n\qa_0$};
\draw (8.25,6.03)--(2.5,7.18);
\draw(4,8.2)--(2.5,7.18)node[midway,left]{\small$p_{r+1}(\qa_1+(k_{r+1}-2)\delta)$~~};
\draw [dashed](2.5,7.18)--(1.5,5.6);
\draw(9.25,4.05)--(1.5,5.6);
\draw(10,2) --(1,3.8)node [midway,below]{\small$N'\qa_0$};
\draw(1.5,5.6) --(1,3.8)node[midway,left]{\small$p_s(\qa_1+(k_s-2)\delta)$};
\end{tikzpicture}
\end{equation}
and we want to prove that $N'=\epsilon_0(e^N_1(b))<N$. But we have\begin{align*}
N'\qa_0&=n\qa_0+\underbrace{\sum_{j=r+1}^sp_j(\qa_0+(k_j-1)\delta)}_{=:u}-\sum_{j=r+1}^sp_j(\qa_1+(k_j-2)\delta)\\
&=\underbrace{n\qa_0+u-\sum_{j=r+1}^sp_j(\qa_1+k_j\delta)}_{=v-N\qa_1}+\underbrace{(\sum_{j=r+1}^s2k_j)}_{=:d}\delta
\end{align*}
where we recall that $v=g\alpha_1+h\delta$ for some $g>0$, $h\ge0$. Thus $N'=d+h$ and $N=d+g+h=N'+g>N'$ as expected. Computations are similar if all diagonals are directed by $\qa_1$. We deal symmetrically with the case $\lambda=\bar\lambda= a_\bullet=\bar a_{\ge2}=0$ and $\bar a_1>0$.

We also prove the other way around by induction. Assume that $P_b$ looks like the previous figure~\ref{polN'}, with $\bar a_1=N'$ (once again, the case $a_1\neq0$ is dealt with similarly). We want to prove that if $N>N'$, $P_{f_1^N(b)}$ satisfies $\lambda=\bar\lambda=\bar a_\bullet=a_{\ge2}=0$ and $a_1=N>0$.
If not, it looks like
\[
\begin{tikzpicture}[scale=0.9]
\draw(10,2) node[right]{\small$\mu^{k_s}=\dots=\mu^\infty$};
\draw(9.25,4.05) node[right]{\small$\mu^{k_{s-1}}=\dots=\mu^{k_s-1}$};
\draw(8.25,6.03) node[right]{\small$\mu^{k_{r+1}}=\dots=\mu^{k_{r+2}-1}$};
\draw(7,7.6) node[right]{~\small$\mu^{k_{r}}=\dots=\mu^{k_{r+1}-1}$};
\draw(10,2)--(9.25,4.05)node[midway,right]{\small$p_s(\qa_0+(k_s-1)\delta)$};
\draw[dashed](8.25,6.03)--(9.25,4.05);
\draw  (7,7.6)--(8.25,6.03)node[midway,right]{~\small$p_{r+1}(\qa_0+(k_{r+1}-1)\delta)$};
\draw (7,7.6)--(4,8.2) node [midway,above]{\small$n\qa_0$};
\draw(10,2) --(-1,4.2)node [midway,below]{\small$N''\qa_0$}node[left]{\small$\bar\mu^\infty$};
\draw(10,2)--(0,0)node [midway,below]{\small$N\qa_1$}node[below]{\small$\mu_0$};
\draw[dashed](0,0)to[bend left](-1,4.2)to[bend left](4,8.2);
\end{tikzpicture}
\]
with $N''\ge N$. But then the MV polytope between the two diagonals oriented by $\qa_0$ has the same right Lusztig datum than $P_b$ represented by~\ref{polN'}, but a different left one as $N''>N'$, which is absurd thanks to~\cite[Theorem 3.11]{BDKT12}.
\end{proof}

\begin{lemm}\label{lzzz}
Assume that $b=f_{i_1}^{k_1}\dots f_{i_t}^{k_t}(b_0)$
where $i_{k+1}=i_k+1\in\Z/2\Z$ and $k_1>\dots>k_t$, and $i_1=1$.
Then for every $K\le k_1$, we have $\epsilon_0(e_1^K(b))<K$.
\end{lemm}

\begin{proof}
We can assume that $P_b$ looks like~\ref{polN}, with $N=k_1$. Then $P_{e_1^N(b)}$ looks like~\ref{polN'} with $N'=k_2$. If $K<N-N'$, we have $b'=e_1^K(b)=f_1^{N-K}f_0^{N'}\dots  f_{i_t}^{k_t}(b_0)$ with $N-K>N'$. Thus satisfies the conditions of~\ref{toponly} and a fortiori $\epsilon_0(b')=0$. If $K=N-N'$ then $P_{b'}$ looks like
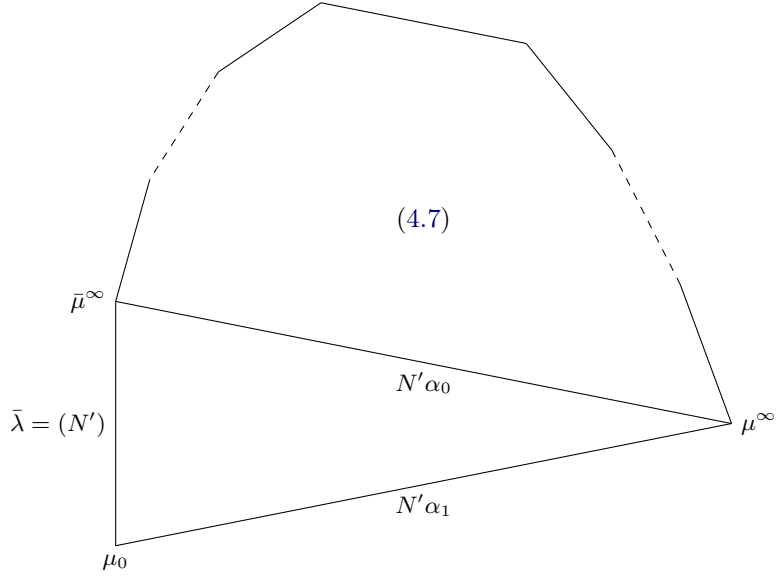
\begin{equation*}
\begin{tikzpicture}[scale=0.9,baseline=(current  bounding  box.center)]
\draw(1,3.8) node[left]{\small$\bar\mu^\infty$};
\draw(10,2) node[right]{\small$\mu^\infty$};
\draw(10,2)--(9.25,4.05)node[midway,right]{};
\draw[dashed](8.25,6.03)--(9.25,4.05);
\draw  (7,7.6)--(8.25,6.03)node[midway,right]{};
\draw (7,7.6)--(4,8.2) node [midway,above]{};
\draw(5.5,5)node{(\ref{polN'})};
\draw(4,8.2)--(2.5,7.18)node[midway,left]{};
\draw [dashed](2.5,7.18)--(1.5,5.6);
\draw(10,2) --(1,3.8)node [midway,below]{\small$N'\qa_0$};
\draw(1.5,5.6) --(1,3.8)node[midway,left]{};
\draw(10,2)--(1,0.2)node [midway,below]{\small$N'\qa_1$}node[below]{\small$\mu_0$};
\draw(1,3.8)--(1,0.2)node[midway,left]{\small$\bar\lambda=(N')$};
\end{tikzpicture}
\end{equation*}
thus again $\epsilon_0(b')=0$. Finally assume that $K>N-N'$ and set $h=N-K$. Then $P_{b'}$ looks like
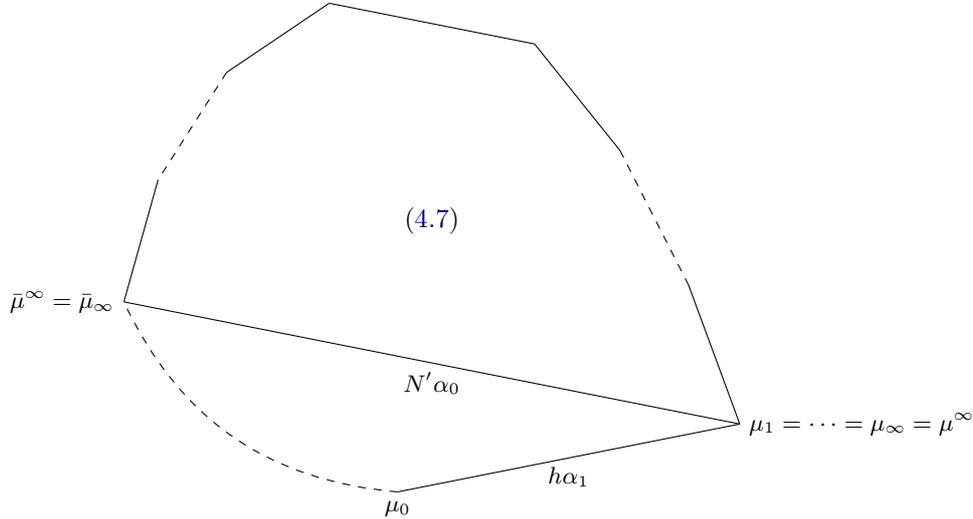
\begin{equation*}
\begin{tikzpicture}[scale=0.9,baseline=(current  bounding  box.center)]
\draw(1,3.8) node[left]{\small$\bar\mu^\infty=\bar\mu_\infty$};
\draw(10,2) node[right]{\small$\mu_1=\dots=\mu_\infty=\mu^\infty$};
\draw(10,2)--(9.25,4.05)node[midway,right]{};
\draw[dashed](8.25,6.03)--(9.25,4.05);
\draw  (7,7.6)--(8.25,6.03)node[midway,right]{};
\draw (7,7.6)--(4,8.2) node [midway,above]{};
\draw(5.5,5)node{(\ref{polN'})};
\draw(4,8.2)--(2.5,7.18)node[midway,left]{};
\draw [dashed](2.5,7.18)--(1.5,5.6);
\draw(10,2) --(1,3.8)node [midway,below]{\small$N'\qa_0$};
\draw(1.5,5.6) --(1,3.8)node[midway,left]{};
\draw(10,2)--(5,1)node [midway,below]{\small$h\qa_1$}node[below]{\small$\mu_0$};
\draw[dashed](5,1)to[bend left](1,3.8);
\end{tikzpicture}
\end{equation*}
where we necessarily have $\epsilon_0<N'-h=N'-N+K<K$ as wished.
\end{proof}

\subsection{$\delta$-top polytopes}

\begin{defi}
We call \emph{$\delta$-top polytope} an MV polytope $P$ such that $a_\bullet=\bar a_\bullet=0$ except $a_1$ or $\bar a_1$.
We denote by $MV^{\delta-t}$ the set of $\delta$-top polytopes.
\end{defi}

\begin{prop}\label{dtoponly}
Consider $b\in \mathcal B(\infty)$ and its MV polytope $P_b$.We have $P_b\in MV^{\delta-t}$ if and only if \[
b=f_{i_1}^{k_1}\dots f_{i_t}^{k_t}(b_0)\]
where $i_{k+1}=i_k+1\in\Z/2\Z$ and $k_1\ge\dots\ge k_t$.
\end{prop}

\begin{proof} 
We use induction on the weight.
Consider $P\in MV^{\delta-t}$. First assume that $\bar\lambda\neq\lambda$, say for instance $\lambda=N\cup\bar\lambda$. Then $P=f_0^N(P')$ where $P'$ is simply the part of $P$ which is above the bottom $\qa_1$-active diagonal. We conclude by induction hypothesis since $\epsilon_1(P')=N$. Now assume that $\bar\lambda=\lambda$ in $P$, with, say, $a_1=N\neq0$. If $N=\lambda_1$, then $e_1^N(P)=P'$ with $(P')^t=P^t$, $\lambda'=\lambda$ and $\bar\lambda'=\lambda\setminus N$. We conclude by induction hypothesis. Otherwise, $N>\lambda_1$. Write $P'=e_1^N(P)$. If $\bar\lambda'\neq\lambda'=\lambda$, we have $\epsilon_0(P')=\lambda_1<N$ and we conclude by induction hypothesis. If $\bar\lambda'=\lambda'=\lambda$, we have $f_1^N(P'^t)=P^t$ by~\cite[Theorem 3.11]{BDKT12}. By~\ref{toponly}, we necessarily have $N>\epsilon_0(P'^t)=\epsilon_0(P')$ and we can again conclude by induction hypothesis.
\end{proof}

\subsection{General polytopes}

Thanks to~\ref{actdiag}, there is a natural map $MV\to MV^{\delta-t}$, $P\mapsto P^{\delta-t}$ which is the removal of the part of a  polytope which is below the active diagonal $[\bar\mu_\infty,\mu_\infty]$.

\begin{prop}\label{genpol}
For every $P\in MV$ there exist $k_1,\dots,k_r >0$ and $h\ge0$ such that $f^h_{i_0}e^{k_1}_{i_1}\dots e^{k_r}_{i_r}(P)=P^{\delta-t}$, where $i_{j+1}=i_j+1\in\Z/2\Z$.
\end{prop}

\begin{proof}
Consider the case where $\bar a_1\neq 0$ in $P^{\delta-t}$. Thanks again to~\cite[Proposition 4.10]{BDKT12} we may assume that the bottom part of our polytope is a stack of $p$ subpolyoptes of the following type
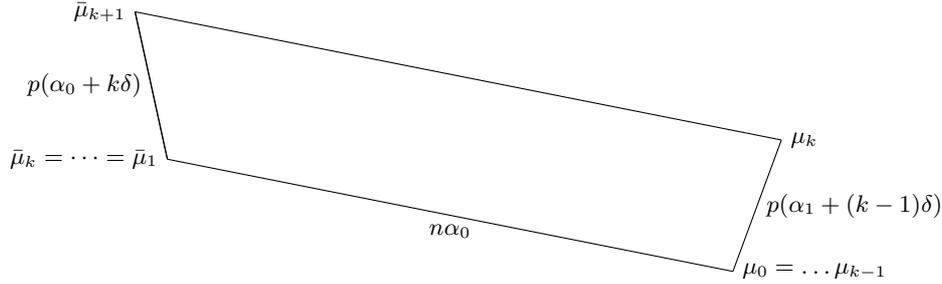
\begin{equation*}
\begin{tikzpicture}[scale=0.85,baseline=(current  bounding  box.center)]
\draw(0,0) node[left]{\small$\bar\mu_{k+1}$};
\draw(0.5,-2.3) node[left]{\small$\bar\mu_{k}=\dots=\bar\mu_1$};
\draw(10,-2) node[right]{\small$\mu_{k}$};
\draw(9.25,-4.05) node[right]{\small$\mu_0=\dots\mu_{k-1}$};
\draw(0.5,-2.3)--(0,0)--(10,-2)node [midway,below]{\small$$};
\draw(9.25,-4.05)--(0.5,-2.3)node [midway,below]{\small$n\qa_0$};
\draw(0.5,-2.3)--(0,0)node[midway,left]{\small$p(\qa_0+k\delta)$};
\draw(10,-2)--(9.25,-4.05)node[midway,right]{\small$p(\qa_1+(k-1)\delta)$};
\end{tikzpicture}
\end{equation*}
where by definition $(\mu_k-\bar\mu_{k-1},\omega_0)\le0$. Thus the line through $\bar\mu_1$ directed by $\qa_1$ meets the edge $[\bar\mu_{k+1},\mu_k]$:
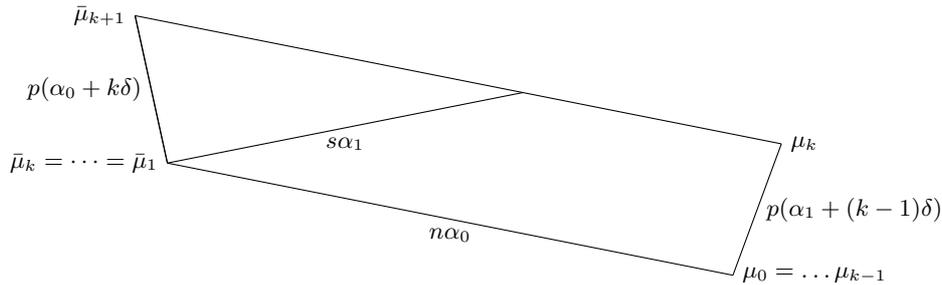
\begin{equation*}
\begin{tikzpicture}[scale=0.85,baseline=(current  bounding  box.center)]
\draw(0,0) node[left]{\small$\bar\mu_{k+1}$};
\draw(0.5,-2.3) node[left]{\small$\bar\mu_{k}=\dots=\bar\mu_1$};
\draw(10,-2) node[right]{\small$\mu_{k}$};
\draw(10,-2) --(0,0)node [pos=0.7,above]{\small$$};
\draw(6,-1.2) --(0.5,-2.3)node [midway,below]{\small$s\qa_1$};
\draw(9.25,-4.05) node[right]{\small$\mu_0=\dots\mu_{k-1}$};
\draw(0.5,-2.3)--(0,0)node [midway,below]{\small$$};
\draw(9.25,-4.05)--(0.5,-2.3)node [midway,below]{\small$n\qa_0$};
\draw(0.5,-2.3)--(0,0)node[midway,left]{\small$p(\qa_0+k\delta)$};
\draw(10,-2)--(9.25,-4.05)node[midway,right]{\small$p(\qa_1+(k-1)\delta)$};
\end{tikzpicture}
\end{equation*}
for some $s>0$.
But then the bottom part of $e_1^se_0^n(P)$ is a stack of $p-1$ polytopes of the previous type and $e_1^se_0^n(P)^{\delta-t}$ and $P^{\delta-t}$ share the same left Lusztig datum except $\bar a_1$. By induction on $p$, there exist $k_1,\dots,k_r >0$ such that $e^{k_1}_{i_1}\dots e^{k_r}_{i_r}(P)$ and $P^{\delta-t}$ share the same left Lusztig datum except $\bar a_1$ and we can conclude since the left Lusztig datum determines the polytope.
\end{proof}

\begin{exem} If the bottom part of $P$ looks like 
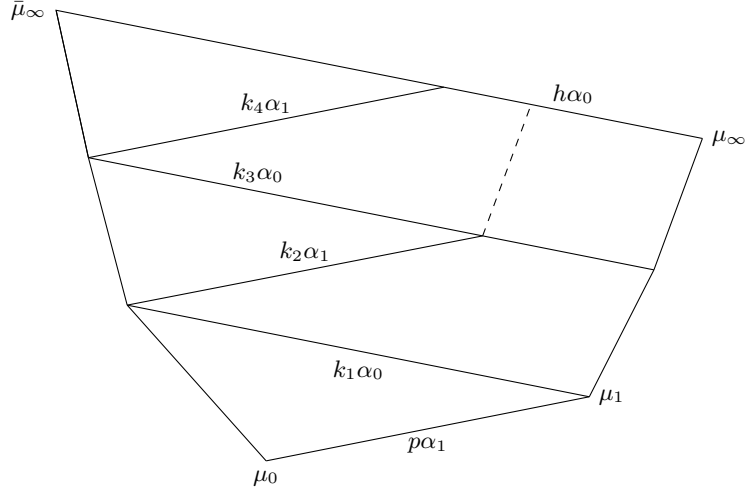
\[
\begin{tikzpicture}[scale=0.85,baseline=(current  bounding  box.center),y=-1cm]
\draw(0,0) node[left]{\small$\bar\mu_\infty$};
\draw(0.5,2.3) node[left]{\small$$};
\draw(1.1,4.6)node[left]{\small$$};
\draw(10,2) node[right]{\small$\mu_\infty$};
\draw(9.25,4.05) node[right]{\small$$};
\draw(8.25,6.03) node[right]{\small$\mu_1$};
\draw(3.25,7.03) node[below]{\small$\mu_0$};
\draw(0.5,2.3)--(0,0)--(10,2)node [pos=0.8,above]{\small$h\qa_0$};
\draw(0.5,2.3)(0,0)(10,2)(9.25,4.05)--(0.5,2.3)node [pos=0.7,above]{\small$k_3\qa_0$};
\draw(0.5,2.3)--(0,0)node[midway,left]{\small$$};
\draw(10,2)--(9.25,4.05)node[midway,right]{\small$$};
\draw[](0.5,2.3)--(1.1,4.6)  (8.25,6.03)--(9.25,4.05);
\draw (1.1,4.6)(2,6.6)(7,7.6)(8.25,6.03)--(1.1,4.6)node[midway,below]{\small$k_1\qa_0$};
\draw (8.25,6.03)--(3.25,7.03) node [midway,below]{\small$p\qa_1$};
\draw (1.1,4.6)--(3.25,7.03) node [midway,above]{\small$$};
\draw (6.6,3.52)--(1.1,4.6) node [midway,above]{\small$k_2\qa_1$};
\draw[dashed] (6.6,3.52)--(7.35,1.47);
\draw (6,1.2)--(0.5,2.3) node [midway,above]{\small$k_4\qa_1$};
\end{tikzpicture}
\]
then $P^{\delta-t}=f_0^he_1^{k_4}e_0^{k_3}e_1^{k_2}e_0^{k_1}e_1^p(P)$. The dashed segment is part of the right side of $e_0^{k_1}e_1^p(P)$.
\end{exem}

\section{Decorations and paths}

In this section we stay in the $\widehat{\mathfrak{sl}}_2$ case. For a given LS path $\pi$ we want to be able to recognize the partitions decorating the associated affine MV polytope. Since we have seen in~\ref{MVretract} that retractions allow one to recover the bottom part of this polytope, we may thanks to~\ref{genpol} stick to the case of paths associated to $\delta$-top polytopes, which we treat in this section.

\begin{defi}\label{defzigzag}
Consider an LS path $\pi$ and $k\in\Z$. An interval $[s,v]\subseteq[0,1]$ is said to be an $\qa_i$-zigzag if there exists $[t,u]\subset[s,v]$ such that $[s,u]$ is an $\qa_i$-stable section at $k$ and $[t,v]$ is an $\qa_{i+1}$-stable section at $-k$. Moreover, we want $\qa_i([s,v])\ge k$ and $\qa_{i+1}([s,v])\ge -k$. 
\end{defi}

\begin{center}
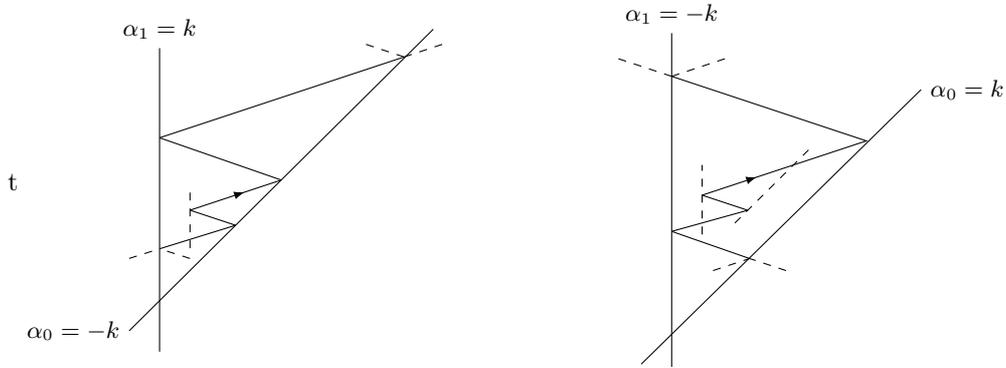

\begin{scaletikzpicturetowidth}{\textwidth}
\begin{tikzpicture}[scale=.8, 
 baseline=(current bounding box.center)]
\node(1)at(-.5,-1){};
\node(2)at(-.5,4.5){\small$\qa_1=k$};
\draw(1)--(2);
\draw[](4,4.5)--(-1,-0.5)node[left]{\small$\qa_0=-k$};
\draw[dashed](0,0.87)--(0,1.8);
\draw[dashed](-1,0.7)--(-.5,.86);
\draw[dashed](0,0.7)--(-.5,.86);
\draw(-.5,.86)--(0.75,1.25);
\draw(0.75,1.25)--(0,1.5);
\draw[middlearrow={latex}](0,1.5)--(1.5,2);
\draw(1.5,2)--(-.5,2.7);
\draw[](-.5,2.7)--(3.52,4.04);
\draw[dashed](3.52,4.04)--(4.21,4.27);
\draw[dashed](3.52,4.04)--(2.83,4.27);
\end{tikzpicture}
\hspace{2cm}
\begin{tikzpicture}[scale=.8, 
 baseline=(current bounding box.center)]
\node(1)at(-.5,-1.5){};
\node(2)at(-.5,4.5){\small$\qa_1=-k$};
\draw(1)--(2);
\draw[dashed](1.75,2.25)--(0.5,1);
\draw[dashed](0,0.85)--(0,2);
\draw(0.75,1.25)--(0,1.5);
\draw(0.75,1.25)--(-.5,0.9);
\draw(-.5,0.9)--(0.78,0.45);
\draw[dashed](1.38,0.25)--(0.78,0.45);
\draw[dashed](.18,0.25)--(0.78,0.45);
\draw[middlearrow={latex}](0,1.5)--(1.5,2);
\draw(1.5,2)--(2.7,2.4)--(-.5,3.47);
\draw[dashed](-.5,3.47)--(-1.4,3.77);
\draw[dashed](-.5,3.47)--(.4,3.77);
\draw(-1,-1.3)--(3.6,3.25)node[right]{\small$\qa_0=k$};
\end{tikzpicture}\end{scaletikzpicturetowidth}
\captionof{figure}{An $\qa_1$-zigzag at $k$ on the left, an $\qa_0$-zigzag at $k$ on the right.}
\end{center}

 

We index partitions by finitely supported sequences of nonnegative integers $(m_k)_{k\ge1}$, the \emph{multiplicities}. The partition associated to such a sequence is denoted in its {exponential form} by $(k^{m_k})$, which means that $k$ appears $m_k$ times, thus giving a partition of $\sum_{k\ge1}m_kk$.

\begin{defi}\label{decodef} Let $\pi$ be an LS path and $k$ an integer. Denote by $m_{i,k}(\pi)$ the number of $\qa_i$-zigzags of $\pi$ at $k$. 
\begin{enumerate}[label=(\roman*)]
\item Assume that $\epsilon_0(\pi)=N>0=\epsilon_1(\pi)$. The \emph{left partition} $\bar\lambda(\pi)$ of $\pi$ is defined by the multiplicities $m_{1,k}(\pi)$, $k>0$. The \emph{right partition} is $\lambda(\pi)=N\cup\bar\lambda(\pi)$ if  $\pi$ reaches $\min\qa_0(\pi)=-N$ within an $\qa_1$-stable section at $N$. Otherwise $\lambda(\pi)=\bar\lambda(\pi)$.
\item Assume that $\epsilon_1(\pi)=N>0=\epsilon_0(\pi)$. The \emph{right partition} $\lambda(\pi)$ of $\pi$ is defined by the multiplicities $m_{0,k}(\pi)$, $k>0$. The \emph{left partition} is $\bar\lambda(\pi)=N\cup\lambda(\pi)$ if  $\pi$ reaches $\min\qa_1(\pi)=-N$ within an $\qa_0$-stable section at $N$. Otherwise $\bar \lambda(\pi)=\lambda(\pi)$.
\end{enumerate}
\end{defi}

\begin{rema} Such assumptions are typically (but not only) satisfied by $\delta$-top polytopes.\end{rema}

\begin{lemm}\label{lzz}
If $\lambda(\pi)=N\cup\bar\lambda(\pi)$ (resp. $\bar\lambda(\pi)=N\cup\lambda(\pi)$), then $\pi$ is in the image of $f_0^Nf_1^N$ (resp. $f_1^Nf_0^N$).
\end{lemm}

\begin{proof} Consider $\pi$ with an $\qa_1$-stable section at $N$, during which $\pi$ reaches $\qa_0=-N=\min\qa_0(\pi)$.  Then $\pi$ is equal to $f_{0}^N(\pi')$ for $\pi'$ reaching $\qa_1=-N$. Thus $\pi'$ is in the image of $f_1^N$.
\end{proof}

\begin{exem} Consider\[
\pi= f_{1}^{\lambda_1} f_{0}^{\lambda_1}\dots f_{1}^{\lambda_r} f_{0}^{\lambda_r}(\pi_{\Lambda})\]
for a partition $\lambda=(\lambda_1\ge\dots\ge\lambda_r)$, $\Lambda$ a large enough dominant integral weight, and $\pi_{\Lambda}$ the corresponding highest weight path.
We get\begin{align*}
\bar\lambda(\pi)&=\lambda\\
\lambda(\pi)&=\lambda\setminus\lambda_1.\end{align*}

Below the examples of $\lambda=(3,2,1)$ on the left and $\lambda=(3,2^2)$ on the right. Red zones correspond to $\qa_1$-stable sections, blue ones to $\qa_0$-stable sections. A red area succeeding a blue one hence gives an $\qa_0$-zigzag. Below each path figures the associated decorated polytope.

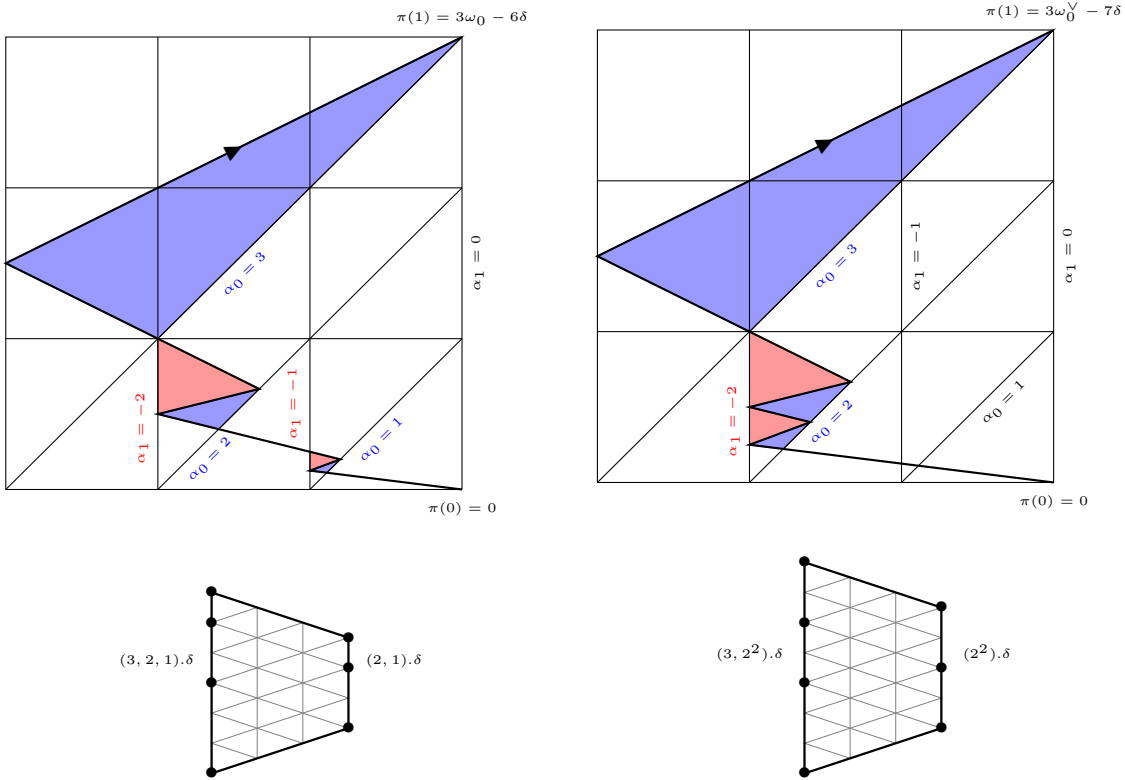
\begin{multicols}{2}
\[
\begin{tikzpicture}[scale = 1]
\draw[fill=blue!40]  (-4,2)--(-6,3)--(0,6) -- cycle;
\draw[fill=blue!40]  (-16/5,4/5)--(-4,1)--(-8/3,4/3)-- cycle;
\draw[fill=red!40]  (-8/3,4/3)--(-4,2)--(-4,1) -- cycle;
\draw[fill=blue!40]  (-16/9,2/9)--(-2,.25)--(-8/5,2/5) -- cycle;
\draw[fill=red!40]  (-2,.5)--(-2,.25)--(-8/5,2/5) -- cycle;
\draw(0,0)--(-6,0)--(-6,6)--(0,6)--(0,0)node[pos=0.5,sloped,below,rotate=180] {\tiny$\alpha_1=0$};
\draw(-4,0)--(-4,6)node[pos=0.135,sloped,above,rotate=0,red] {\tiny$\alpha_1=-2$};
\draw(-6,4)--(0,4);
\draw(-2,0)--(-2,6)node[pos=0.19,rotate=0,sloped,above,red] {\tiny$\alpha_1=-1$};
\draw(0,2)--(-6,2);
\draw(-4,0)--(0,4)node[pos=.135,sloped,below,rotate=0,blue] {\tiny$\alpha_0=2$};
\draw(-2,0)--(0,2)node[pos=0.4,sloped,below,rotate=0,blue] {\tiny$\alpha_0=1$};
\draw(-6,0)--(0,6)node[midway,sloped,below,rotate=0,blue] {\tiny$\alpha_0=3$};
\draw[thick](0,0)--(-2,.25)--(-8/5,2/5)--(-4,1)--(-8/3,4/3)--(-6,3)--(0,6)node[pos=0.5,sloped]{$\blacktriangleright$};
\draw(.,-.25) node{\tiny $\pi(0)=0$};
\draw(.,6.25) node{\tiny $\pi(1)=3\omega_0-6\delta$};
\end{tikzpicture}
\]
\[
\begin{tikzpicture}[yscale=0.2, xscale=0.6]
\draw [line width = 0.01cm, color=gray] (1,1)--(0,0)--(3,-3)--(0,-6)--(2,-8);
\draw [line width = 0.01cm, color=gray] (1,-9)--(1,1);
\draw [line width = 0.01cm, color=gray] (2,-8)--(2,0);
\draw [line width = 0.01cm, color=gray] (2,0)--(0,-2)--(3,-5)--(0,-8)--(1,-9);
\draw [line width = 0.01cm, color=gray] (3,-1)--(0,-4)--(3,-7);
\draw (3,-1) node {$\bullet$};
\draw (3,-3) node {$\bullet$};
\draw (3,-7) node {$\bullet$};
\draw (0,2) node {$\bullet$};
\draw (0,0) node {$\bullet$};
\draw(0,-4)node{$\bullet$};
\draw(0,-10)node{$\bullet$};
\draw  [line width = 0.03cm] (3,-1)--(0,2)--(0,-10)--(3,-7)--(3,-1);
\draw (-1.2,-2.5) node {\tiny {$(3,2,1)$}.$\delta$};
\draw (4,-2.5) node {\tiny {$(2,1)$}.$\delta$};
\end{tikzpicture}
\]

\columnbreak

\[
\begin{tikzpicture}[scale = 1]
\draw[fill=blue!40]  (-4,2)--(-6,3)--(0,6) -- cycle;
\draw[fill=blue!40]  (-16/5,4/5)--(-4,1)--(-8/3,4/3)-- cycle;
\draw[fill=blue!40]  (-32/9,4/9)--(-4,.5)--(-16/5,4/5) -- cycle;
\draw[fill=red!40]  (-4,.5)--(-16/5,4/5)--(-4,1) -- cycle;
\draw[fill=red!40]  (-8/3,4/3)--(-4,2)--(-4,1) -- cycle;
\draw(0,0)--(-6,0)--(-6,6)--(0,6)--(0,0)node[pos=0.5,sloped,below,rotate=180] {\tiny$\alpha_1=0$};
\draw(-4,0)--(-4,6)node[pos=0.135,sloped,above,rotate=0,red] {\tiny$\alpha_1=-2$};
\draw(-6,4)--(0,4);
\draw(-2,0)--(-2,6)node[pos=0.5,sloped,below,rotate=0] {\tiny$\alpha_1=-1$};
\draw(0,2)--(-6,2);
\draw(-4,0)--(0,4)node[pos=0.235,sloped,below,rotate=0,blue] {\tiny$\alpha_0=2$};
\draw(-2,0)--(0,2)node[pos=0.6,sloped,below,rotate=0] {\tiny$\alpha_0=1$};
\draw(-6,0)--(0,6)node[midway,sloped,below,rotate=0,blue] {\tiny$\alpha_0=3$};
\draw[thick](0,0)--(-4,.5)--(-16/5,4/5)--(-4,1)--(-8/3,4/3)--(-6,3)--(0,6)node[pos=0.5,sloped]{$\blacktriangleright$};
\draw(.,-.25) node{\tiny $\pi(0)=0$};
\draw(.,6.25) node{\tiny $\pi(1)=3\omega_0^\vee-7\delta$};
\end{tikzpicture}
\]
\[
\begin{tikzpicture}[yscale=0.2, xscale=0.6]
\draw [line width = 0.01cm, color=gray] (2,2)--(0,0)--(3,-3)--(0,-6)--(2,-8);
\draw [line width = 0.01cm, color=gray] (1,-9)--(1,3);
\draw [line width = 0.01cm, color=gray] (2,-8)--(2,2);
\draw [line width = 0.01cm, color=gray] (3,1)--(0,-2)--(3,-5)--(0,-8)--(1,-9);
\draw [line width = 0.01cm, color=gray] (1,3)--(0,2)--(3,-1)--(0,-4)--(3,-7);
\draw (3,1) node {$\bullet$};
\draw (3,-3) node {$\bullet$};
\draw (3,-7) node {$\bullet$};
\draw (0,4) node {$\bullet$};
\draw (0,0) node {$\bullet$};
\draw(0,-4)node{$\bullet$};
\draw(0,-10)node{$\bullet$};
\draw  [line width = 0.03cm] (3,1)--(0,4)--(0,-10)--(3,-7)--(3,1);
\draw (-1.2,-2) node {\tiny {$(3,2^2)$}.$\delta$};
\draw (4,-2) node {\tiny {$(2^2)$}.$\delta$};
\end{tikzpicture}
\]
\end{multicols}
\end{exem}

\begin{lemm}\label{stsechere}
Let $\pi$ be an LS path and consider integers $1\le p\le n$ such that $f_1^n(\pi)$ exists. We have the following.\begin{enumerate}[label=(\roman*)]
\item Assume that $\epsilon_0(\pi)=p$, so that $\min\alpha_0(\pi)=-p$. If $\pi$ reaches this minimum during an $\qa_1$-stable section at $p$, then $f_1^n(\pi)$ has an $\qa_0$-zigzag at $p$.
\item Assume that $\epsilon_0(\pi)=n$. Then $f_1^n(\pi)$ has an $\qa_0$-stable section at $n$, during which $\qa_1(f_1^n(\pi))$ reaches its minimum $-n$.
\item Assume that $\epsilon_0(\pi)=p$, and consider a positive integer $q$. Then  $\pi$ has an $\qa_1$-zigzag at $q$ if and only if $f_1^n(\pi)$ has an $\qa_0$-zigzag at $q$ (in which case $q\le p$).\end{enumerate}
\end{lemm}

\begin{proof}
 For (i), $\pi$ locally looks like
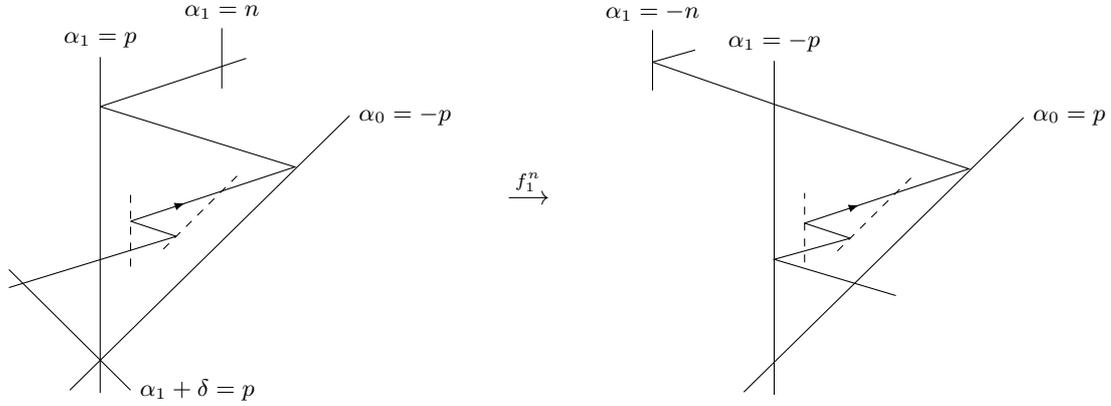
\begin{center}
\begin{tikzpicture}[scale=.8, 
 baseline=(current bounding box.center)]
\node(1)at(-.5,-1.5){};
\node(2)at(-.5,4.5){\small$\qa_1=p$};
\draw(1)--(2);
\draw(-2,.7)--(0,-1.3)node[right]{\small$\qa_1+\delta=p$};
\draw[dashed](1.75,2.25)--(0.5,1);
\draw[dashed](0,0.75)--(0,2);
\draw(-2,0.4)--(0.75,1.25);
\draw(0.75,1.25)--(0,1.5);
\draw[middlearrow={latex}](0,1.5)--(1.5,2);
\draw(1.5,2)--(2.7,2.4)--(-.5,3.4);
\draw[](-.5,3.4)--(1.9,4.2);
\draw(1.5,3.7)--(1.5,4.7)node[above]{\small$\qa_1=n$};
\draw(-1,-1.3)--(3.6,3.25)node[right]{\small$\qa_0=-p$};
\end{tikzpicture}
\hspace{0.5cm}$\overset{f_1^n}{\longrightarrow}$\hspace{0.5cm}
\begin{tikzpicture}[scale=.8, 
 baseline=(current bounding box.center)]
\node(1)at(-.5,-1.5){};
\node(2)at(-.5,4.5){\small$\qa_1=-p$};
\draw(1)--(2);
\draw[dashed](1.75,2.25)--(0.5,1);
\draw[dashed](0,0.85)--(0,2);
\draw(0.75,1.25)--(0,1.5);
\draw(0.75,1.25)--(-.5,0.9)--(1.5,0.3);
\draw[middlearrow={latex}](0,1.5)--(1.5,2);
\draw(1.5,2)--(2.7,2.4)--(-.5,3.47);
\draw[](-.5,3.47)--(-2.5,4.17);
\draw[](-2.5,4.17)--(-1.8,4.37);
\draw(-2.5,3.7)--(-2.5,4.7)node[above]{\small$\qa_1=-n$};
\draw(-1,-1.3)--(3.6,3.25)node[right]{\small$\qa_0=p$};
\end{tikzpicture}\end{center}
and we have an $\qa_0$-zigzag at $p$: indeed, it can not be included in a larger $\qa_0$-stable one because if $\pi$ hits $\qa_0=-p$ after $\qa_1=n$, then after applying $f_1^n$ this hit is at $\qa_0=2n-p\ge p$. The same picture (without $\qa_1$-stable sections) proves (ii), merging the two walls $\qa_1=p$ and $\qa_1=n$.
 For (iii), $\pi$ locally looks like
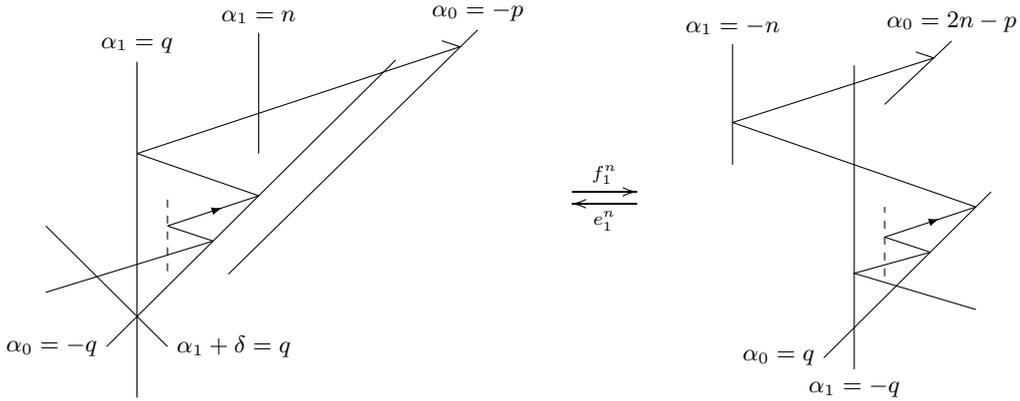
\begin{center}
\begin{tikzpicture}[scale=.8, 
 baseline=(current bounding box.center)]
\node(1)at(-.5,-1.5){};
\node(2)at(-.5,4.5){\small$\qa_1=q$};
\draw(1)--(2);
\draw(-2,1.5)--(0,-.5)node[right]{\small$\qa_1+\delta=q$};
\draw[](3.75,4.25)--(-1,-0.5)node[left]{\small$\qa_0=-q$};
\draw[dashed](0,0.75)--(0,2);
\draw(-2,0.4)--(0.75,1.25);
\draw(0.75,1.25)--(0,1.5);
\draw[middlearrow={latex}](0,1.5)--(1.5,2);
\draw(1.5,2)--(-.5,2.7);
\draw[](-.5,2.7)--(4.81,4.47)--(4.51,4.57);
\draw(1.5,2.7)--(1.5,4.7)node[above]{\small$\qa_1=n$};
\draw(1,.7)--(5.1,4.75)node[above]{\small$\qa_0=-p$};
\end{tikzpicture}
\hspace{0.2cm}$\xymatrix{{}\ar@<0.5ex>[r]^-{f_1^n}&{}\ar@<0.5ex>[l]^-{e_1^n}}$\hspace{0.2cm}
\begin{tikzpicture}[scale=.8, 
 baseline=(current bounding box.center)]
\node(1)at(-.5,-1){\small$\qa_1=-q$};
\node(2)at(-.5,4.5){};
\draw(1)--(2);
\draw[](1.75,2.25)--(-1,-0.5)node[left]{\small$\qa_0=q$};
\draw[dashed](0,0.85)--(0,2);
\draw(0.75,1.25)--(0,1.5);
\draw[middlearrow={latex}](0,1.5)--(1.5,2);
\draw(0.75,1.25)--(-.5,0.9)--(1.5,0.3);
\draw(1.5,2)--(-.5,2.7);
\draw[](-.5,2.7)--(-2.5,3.4)--(.81,4.47)--(.51,4.57);
\draw(-2.5,2.7)--(-2.5,4.7)node[above]{\small$\qa_1=-n$};
\draw(0,3.7)--(1.1,4.75)node[above]{\small$\qa_0=2n-p$};
\end{tikzpicture}
\end{center}
where the dashed wall might be at $\qa_1=q$ ($-q$ after applying $f_1^n$). Since $2n-p\ge q$, we get what we need.
\end{proof}

\begin{coro}
Consider $b$ such that $P_b\in MV^t$.
Then the corresponding path $\pi$ satisfies $\bar\lambda(\pi)=\lambda(\pi)=0=\bar\lambda=\lambda$.
\end{coro}

\begin{proof}
We use induction on the weight. Assume that $\lambda(\pi)$ or $\bar\lambda(\pi)$ is nonzero. Then by~\ref{lzz} and~\ref{lzzz} we have $\lambda(\pi)=\bar\lambda(\pi)$. Assume then that $\pi$ has an $\qa_0$-zigzag at $k>0$ (the case of an $\qa_1$-zigzag is dealt with symmetrically). By induction hypothesis, if $k_1=\epsilon_1(\pi)$, $e_1^{k_1}(\pi)=e_1^{\max}(\pi)$ does not have any $\qa_1$-zigzag at $k$. It implies that all $(-\qa_0)$-directed sections in $e_1^{\max}(\pi)$ occur before the $\qa_1$-stable section at $k$ induced by the $\qa_1$-stable section at $-k$ in the $\qa_0$-zigzag of $\pi$. On the right hand side here:
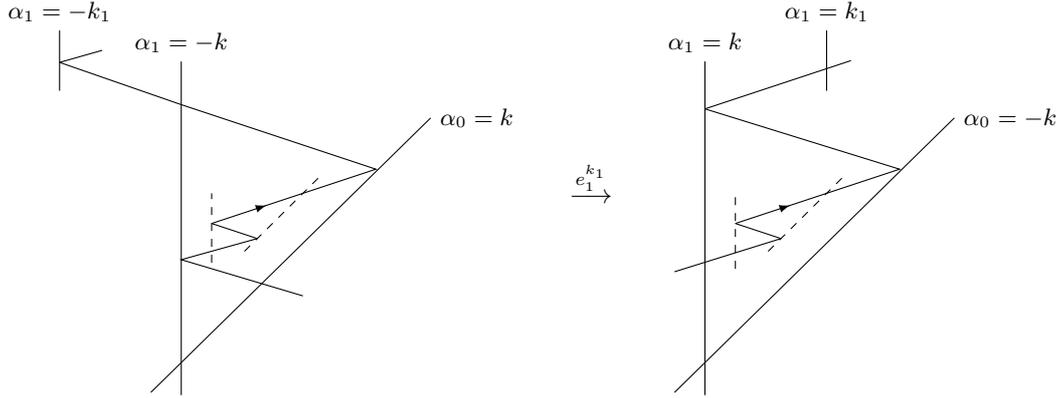
\begin{center}
\begin{tikzpicture}[scale=.8, 
 baseline=(current bounding box.center)]
\node(1)at(-.5,-1.5){};
\node(2)at(-.5,4.5){\small$\qa_1=-k$};
\draw(1)--(2);
\draw[dashed](1.75,2.25)--(0.5,1);
\draw[dashed](0,0.85)--(0,2);
\draw(0.75,1.25)--(0,1.5);
\draw(0.75,1.25)--(-.5,0.9)--(1.5,0.3);
\draw[middlearrow={latex}](0,1.5)--(1.5,2);
\draw(1.5,2)--(2.7,2.4)--(-.5,3.47);
\draw[](-.5,3.47)--(-2.5,4.17);
\draw[](-2.5,4.17)--(-1.8,4.37);
\draw(-2.5,3.7)--(-2.5,4.7)node[above]{\small$\qa_1=-k_1$};
\draw(-1,-1.3)--(3.6,3.25)node[right]{\small$\qa_0=k$};
\end{tikzpicture}
\hspace{0.5cm}$\overset{e_1^{k_1}}{\longrightarrow}$\hspace{0.5cm}
\begin{tikzpicture}[scale=.8, 
 baseline=(current bounding box.center)]
\node(1)at(-.5,-1.5){};
\node(2)at(-.5,4.5){\small$\qa_1=k$};
\draw(1)--(2);
\draw[dashed](1.75,2.25)--(0.5,1);
\draw[dashed](0,0.75)--(0,2);
\draw(-1,0.7)--(0.75,1.25);
\draw(0.75,1.25)--(0,1.5);
\draw[middlearrow={latex}](0,1.5)--(1.5,2);
\draw(1.5,2)--(2.7,2.4)--(-.5,3.4);
\draw[](-.5,3.4)--(1.9,4.2);
\draw(1.5,3.7)--(1.5,4.7)node[above]{\small$\qa_1=k_1$};
\draw(-1,-1.3)--(3.6,3.25)node[right]{\small$\qa_0=-k$};
\end{tikzpicture}
\end{center}
the wall $\qa_0=-k$ can not be reached after the $\qa_1=k_1$ one, hence $\min\qa_0(e_1^{\max}(\pi))\le-k$ is reached before or during the $\qa_1$-stable section at $k$.
But then $e_0^ke_1^{\max}(\pi)$ reaches $\qa_1=-k$ and $e_1^ke_0^k(e_1^{\max}(\pi))$ exists, which contradicts~\ref{lzzz} applied to $e_1^{\max}(\pi)$.
\end{proof}

\begin{theo}\label{decopath}
Consider $b$ such that $P_b\in MV^{\delta-t}$, and $\pi$ the associated path. Then $\bar\lambda=\bar\lambda(\pi)$ and $\lambda=\lambda(\pi)$. Precisely:\begin{enumerate}[label=(\roman*)]
\item if $a_1\neq0$, then $k\in\bar\lambda\cap\lambda$ if and only if $\pi$ has an $\qa_0$-zigzag at $k$;
\item if $\bar a_1\neq0$, then $k\in\bar\lambda\cap\lambda$ if and only if $\pi$ has an $\qa_1$-zigzag at $k$;
\item $n=\bar\lambda\setminus\lambda$ if and only if $-n=\min\qa_1(\pi)$ is reached during an $\qa_0$-stable section at $n$;
\item $n=\lambda\setminus\bar\lambda$ if and only if $-n=\min\qa_0(\pi)$ is reached during an $\qa_1$-stable section at $n$.
\end{enumerate}
\end{theo}

\begin{proof} We proceed by induction on the weight using~\ref{stsechere}. Assume $a_1=n\neq0$ so that we may write $\pi=f^n_1(\pi')$, and $P'=P_{\pi'}$. The first case is $\bar\lambda=n\cup\lambda$. Then $\bar\lambda'=\lambda'=\lambda$, $\bar a'_1=n$ and $a'_1=0$. Then the induction hypothesis (ii) and~\ref{stsechere}(ii,iii) prove (i,iii) here. Then we consider the second case where $\bar\lambda=\lambda$. There are two subcases. If $\bar\lambda'=\lambda'=\lambda$, we prove (i) thanks to the induction hypothesis (ii) and~\ref{stsechere}(iii). Otherwise $\lambda=p\cup\bar\lambda'$, where $p\le n$ thanks to~\ref{dtoponly}, $\lambda'=\lambda$ and $\bar a'_1=p$. The induction hypothesis (ii) allow us to use~\ref{stsechere}(iii) to prove that $\bar\lambda'$ and $\bar\lambda(\pi)$ differ by at most one element. The induction hypothesis (iv) together with~\ref{stsechere}(i) (applicable since $p\le n$) proves that $p\in\bar\lambda(\pi)$. Thus $\bar\lambda=\bar\lambda(\pi)$. We have proved (i) and (iii) assuming $a_1\neq0$. Similarly, we prove (ii) and (iv) assuming $\bar a_1\neq0$ using  the analog of~\ref{stsechere} where the roles of $\qa_0$ and $\qa_1$ are exchanged.
\end{proof}


\end{document}